\documentclass[reqno, 12pt]{amsart}
\usepackage{amsthm, amstext, amsmath, amssymb, graphicx, verbatim, fullpage, bm}
\usepackage{booktabs}
\usepackage{hyperref}
\usepackage{ mathrsfs }
\usepackage{algorithm}
\usepackage{versions}
\usepackage{enumerate}
\usepackage{graphicx}
\usepackage{mathtools}
\usepackage{algorithm}
\usepackage{tabularx}
\usepackage{tikz-cd}
\usepackage{fullpage, microtype, colonequals}

\usepackage[utf8]{inputenc}
\usepackage[english]{babel}
\usepackage{xcolor}

\usetikzlibrary{matrix}

\newtheorem{thm}{Theorem}[section]
\newtheorem{lemma}[thm]{Lemma} \newtheorem{cor}[thm]{Corollary}
\newtheorem{prop}[thm]{Proposition}
\theoremstyle{definition}
\newtheorem{defn}[thm]{Definition}
\newtheorem{question}[thm]{Question}

\newtheorem{example}[thm]{Example}

\newtheorem{remark}[thm]{Remark}

\newcommand{\A}{\mathbb{A}}
\newcommand{\Z}{\mathbb Z}
\newcommand{\N}{\mathbb N}
\newcommand{\R}{\mathbb R}

\newcommand{\Q}{\mathbb Q}

\newcommand{\shO}{\mathcal{O}}
\newcommand{\shL}{\mathcal{L}}
\newcommand{\PP}{\mathbb{P}}

\newcommand{\G}{\mathbb{G}}

\newcommand{\GITQ}{/ \! /}
\newcommand{\mc}{\mathcal}

\DeclareMathOperator{\diag}{diag}

\DeclareMathOperator{\GL}{GL}
\DeclareMathOperator{\SL}{SL}

\DeclareMathOperator{\Mat}{Mat}

\DeclareMathOperator{\Corner}{Hull_+}
\DeclareMathOperator{\Conv}{Conv}

\DeclareMathOperator{\Span}{Span}

\newcommand{\PE}{\PP(\LEnd_{N+1})}

\DeclareMathOperator{\Prof}{Prof}

\newcommand*\subtxt[1]{_{\textnormal{#1}}}
\newcommand*\suptxt[1]{^{\textnormal{#1}}}

\newcommand{\Vstable}{V\suptxt{s}}
\newcommand{\Vsemistable}{V\suptxt{ss}}
\newcommand{\Vunstable}{V\suptxt{us}}
\newcommand{\gits}{\suptxt{s}}
\newcommand{\gitss}{\suptxt{ss}}

\newcommand{\pathd}{\bm{\downarrow}}
\newcommand{\pathr}{\bm{\rightarrow}}
\newcommand{\Hess}{h}

\DeclareMathOperator{\Verts}{Vert}
\newcommand{\VertPlusI}{{\Verts}_+^I}

\DeclareMathOperator{\Ctrl}{Vert_+}
\DeclareMathOperator{\wt}{wt}

\newcommand{\Gm}{\G\subtxt{m}}
\newcommand{\Twip}{\mathcal{X}_{N,n}}
\newcommand{\Twipq}{\mathcal{X}_{N,n} \GITQ \SL_{N+1}}

\newcommand{\LEnd}{\Mat}

\newcommand{\piv}{\multicolumn{1}{|c}{\bm{*}}}
\newcommand{\stairzerobar}{\multicolumn{1}{c|}{0}}
\newcommand{\stairbarast}{\multicolumn{1}{|c}{*}}
\newcommand{\stairbarzero}{\multicolumn{1}{|c}{0}}
\newcommand{\stairbarvdots}{\multicolumn{1}{|c}{\vdots}}
\newcommand{\stairvdotsbar}{\multicolumn{1}{c|}{\vdots}}

\newcommand{\barwt}{\overline{\wt}}

\newcommand{\Orth}{O^+}
\newcommand{\OrthInt}{(O^+)^\circ}

\title[GIT of linear maps with marked points]{GIT stability of linear maps on projective space with marked points}
\date{\today}
\author{Max Weinreich}
\address{Department of Mathematics\\
Brown University\\
Providence, RI 02906}
\email{maxhweinreich@gmail.com}

\thanks{The author was supported by a National Science Foundation Graduate Research Fellowship under Grant No. 2040433.}

\begin{document}

\maketitle

\begin{abstract}
    We construct moduli spaces of linear self-maps of $\PP^N$ with marked points, up to projective equivalence. That is, we let $\SL_{N+1}$ act simultaneously by conjugation on projective linear maps and diagonally on $(\PP^N)^n$, and we take the geometric invariant theory (GIT) quotient. These moduli spaces arise in algebraic dynamics and integrable systems. Our main result is a dynamical characterization of the GIT semistable and stable loci in the space of linear maps $T$ with marked points. We show that GIT stability can be checked by counting the marked points on flags with certain Hessenberg functions relative to $T$. The proof is combinatorial: to describe the weight polytopes for this action, we compute the vertices and facets of certain convex polyhedra generated by roots of the $A_N$ lattice.
\end{abstract}

\section{Introduction} \label{sect_intro_chapgit}

In this paper, we use geometric invariant theory (GIT) to construct moduli spaces $\mc{M}^N_{1,n}$ of linear self-maps of $\PP^N$ with $n$ marked points, considered up to projective equivalence.

\begin{defn}
Let $N, n \in \N.$ A \emph{linear map on $\PP^N$ with $n$ marked points}, or \emph{marked linear map}, is a tuple $(T, v_1, \hdots, v_n),$ where $T \colon \PP^N \dashrightarrow \PP^N$ is the projectivization of a linear map and $v_1, \hdots, v_n \in \PP^N$. 
\end{defn}

Marked linear maps arise in algebraic dynamics, combinatorial algebraic geometry, and integrable systems; see Sections \ref{sect_alg_dyn}-\ref{sect_integrable_systems} for these connections. Our primary motivation is the problem of understanding moduli of algebraic dynamical systems \cite{MR4007163, MR4366860, MR4132597, MR2741188, MR2884382}. Doyle-Silverman constructed moduli spaces $\mc{M}_{d,n}^N$ of degree $d \geq 2$ dynamical systems on $\PP^N$ with $n$ marked points \cite{MR4132597}. These moduli spaces are used to study dynamical systems with marked periodic orbits, but little is known about their structure. We construct these moduli spaces in the linear case $d = 1$. 

We find that, although linear maps are the simplest algebraic dynamical systems, marking points leads to a remarkably subtle moduli space. We show that $\mc{M}^N_{1,n}$ exists, is rational, and admits a dynamically meaningful compactification $\bar{\mc{M}}^N_{1,n}$ with a delicate combinatorial structure. In this way, our moduli problem is loosely analogous to the widely-studied moduli space of curves of genus $0$ with $n$ marked points.

We work over a fixed algebraically closed field $k$. Fix integers $N, n \in \N$. Let $\LEnd_{N+1}$ be the space of linear endomorphisms of the vector space $k^{N+1}$, and let $\PP(\LEnd_{N+1}) = \PP^{N^2 + 2 N}$ be its projectivization. In coordinates, the elements of $\PP(\LEnd_{N+1})$ are $(N+1) \times (N+1)$ nonzero matrices, up to scale. We consider the variety
$$\Twip \colonequals \PP(\LEnd_{N+1}) \times (\PP^N)^n,$$
equipped with the following action of $\SL_{N+1}$: for each $A \in \SL_{N+1}$,
\begin{equation} \label{eq_acn}
    A \cdot (T,v_1, \hdots, v_n) \colonequals (ATA^{-1}, Av_1, \hdots, Av_n).
\end{equation}

The GIT quotient of $\Twip$ by the action \eqref{eq_acn} is defined relative to a choice of very ample, $\SL_{N+1}$-linearized invertible sheaf $\shL$. Note that any very ample invertible sheaf $\shL$ has a unique $\SL_{N+1}$-linearization. We write $(\Twipq)(\shL)$ for the GIT quotient of $\Twip$ by the action \eqref{eq_acn} relative to $\shL$.

Our main objects of study are the GIT quotients for various sheaves $\shL$,
$$\mc{M}_{1, n}^N(\shL) \colonequals (\Twipq)(\shL).$$

The price of working with GIT is that some geometric objects are so degenerate that they cannot be represented in a quotient variety. Points of $\Twip$ are classified as \emph{unstable}, \emph{semistable}, or \emph{stable} relative to the sheaf $\shL$. Stable points have the desirable property that their orbits are in bijection with the points of a geometric quotient variety. The semistable points are certain degenerations of families of stable points, and they can be used to construct a natural compactification of the stable quotient.

Generally, a geometric object in $\PP^N$ is GIT unstable if it has too high an order of contact with some flag in $\PP^N$. The classic example is due to Mumford \cite{MR1304906}. A configuration $(v_1, \hdots, v_n) \in (\PP^N)^n$ of $n$ points in $\PP^N$ is semistable relative to the diagonal action of $\SL_{N+1}$ and the sheaf $\shO(1, \hdots, 1)$ if, for each nontrivial proper subspace $H \subsetneq k^{N+1}$,
\begin{equation} \label{eq_mumford_template}
\# \{i  \colon v_i \in H \} \leq \frac{n}{N+1}(\dim H).
\end{equation}

Our main theorem characterizes the GIT stable and semistable points of $\Twip$, in the spirit of Mumford's example. We show that the GIT stability of a point $(T, v_1, \hdots, v_n) \in \Twip$ can be determined by counting the points $v_i$ lying on certain flags that play a special dynamical role relative to $T$. To state the theorem, we define an invariant that encapsulates the dynamics of $T$ on a given flag.

\begin{defn}
\label{defn:flaginkN1}
A \emph{flag~$\mc{H}$ in $k^{N+1}$} is a nested sequence of linear subspaces
\[
\mc{H} :
0 \subsetneq H_1 \subsetneq \hdots \subsetneq H_\gamma \subsetneq  k^{N+1}.
\]
For notational convenience, we extend the list of subspaces in~$\mc{H}$ by setting
\[
H_0 = \{0\}, \quad H_{\gamma+1} = k^{N+1}.
\]
Given a flag $\mc{H}$ and $T \in \PE$, the \emph{Hessenberg function of $\mc{H}$ relative to $T$} is defined by
\begin{align*}
    \Hess_{T,\mc{H}}  : \; & \{0, \hdots, \gamma + 1\} \to \{0, \hdots \gamma + 1\}, \\
    & i \mapsto \min \{j : TH_i \subseteq H_j\}.
\end{align*}
Given a flag $\mc{H}$ and a point $T\in\PE$, we say that~$\mc{H}$ is of \emph{Type~I,~II, or~III relative to~$T$} if it has the following properties, although we note that~$\mc{H}$ need not satisfy any of these conditions\textup:
\begin{itemize}
    \item \textup{Type I:}  $\gamma = 1$, and $\mc{H} = (H_1)$ satisfies either
    \[
    \Bigl( 0 \neq TH_1 \subseteq H_1 \Bigr)
    \qquad\text{or}\qquad
    \Bigl(TH_1 \subseteq H_1 \quad \text{and} \quad T(k^{N+1}) \not\subseteq H_1
    \Bigr).
    \]
    \item \textup{Type II:} The Hessenberg function $\Hess_{T, \mc{H}}$ satisfies, for all $t$ in the range $1 \leq t \leq \gamma,$
$$\Hess_{T, \mc{H}}(t) = t + 1.$$
    \item \textup{Type III:} The Hessenberg function $\Hess_{T, \mc{H}}$ satisfies, for all $t$ in the range $1 \leq t \leq \gamma + 1$,
$$\Hess_{T, \mc{H}}(t) = t - 1.$$
\end{itemize}
\end{defn}
Intuitively speaking, Type II flags are unfurled by $T$ in $\gamma + 1$ steps, and Type III flags are furled by $T$ in $\gamma + 1$ steps. A map admits a Type III flag if and only if it is nilpotent. Type I flags are simply those with $\gamma = 1$ that are not Type II or Type III.

Our main theorem shows that GIT stability of a marked linear map can be checked on these three types of flags.
\begin{thm}
\label{thm_main}
Let $k$ be an algebraically closed field, and let $N, n, q \in \N$. We consider the variety 
\[
\Twip = \PE \times (\PP^N)^n
\]
equipped with the action \eqref{eq_acn} of $\SL_{N+1}$, the sheaf $\shL = \shO(q, 1, \hdots, 1)$, and the unique $\SL_{N+1}$-linearization. 
\begin{itemize}
    \item[\textup{(a)}]
    Let $T\in\PE$, and let $\mc{H}$ be a flag in $k^{N+1}$. Then $\mc{H}$ belongs to at most one of the three types I, II, or III relative to~$T$ described in Definition~\ref{defn:flaginkN1}.
    \item[\textup{(b)}]
    A point $(T,v_1, \hdots, v_n) \in \Twip$ is GIT semistable if and only if, for every $1\le\gamma\le N$ and every flag $\mc{H}=(H_1,\ldots,H_\gamma)$ in $k^{N+1}$ that is of Type~I,~II, or~III relative to~$T$ as described in Definition~\ref{defn:flaginkN1}, we have
\begin{equation} \label{eq_main}
\sum_{j=1}^\gamma \# \{i \colon v_i \in H_j\} \leq \frac{n}{N+1}  \sum_{j=1}^\gamma \dim H_j + 
\begin{cases}
0 &\text{if $\mc{H}$ is Type I relative to $T$,} \\
q &\text{if $\mc{H}$ is Type II relative to $T$,} \\
-q &\text{if $\mc{H}$ is Type III relative to $T$.} \\
\end{cases}
\end{equation}
    \item[\textup{(c)}]
    A point $(T,v_1, \hdots, v_n) \in \Twip$ is GIT stable if and only if, for every flag $\mc{H}$, the inequality \eqref{eq_main} in \textup{(b)} is strict.
\end{itemize}
\end{thm}

We deduce Theorem \ref{thm_main} as a special case of Theorem \ref{thm_main_any_sheaf}, which describes GIT stability and semistability relative to any very ample invertible sheaf $\shL = \shO(q, m_1, \hdots, m_n)$, where $q, m_1, \hdots, m_n \in \N$. All our results are characteristic-independent.

\begin{example} \label{ex_one_marked_point}
When the number of marked points is $n = 1$, the complicated criteria of Theorem \ref{thm_main} condense as follows. The stable and semistable loci are equal. The pair $(T,v)$ is stable if $T$ is non-nilpotent and $v$ is a cyclic vector for $T$ up to scaling. The stable quotient is a weighted projective space $\PP(1,2,\ldots N+1)$, related to the Mumford-Suominen moduli space of linear maps equipped with a cyclic vector \cite{MR0437531}. See Section \ref{sect_one_marked_point} for details.
\end{example}

\begin{example} \label{ex_dim_one}
On $\PP^1$, complete flags are in 1-1 correspondence with points. If $T$ is non-nilpotent, then fixed points and indeterminate points correspond to Type I flags and all other points correspond to Type II flags. If $T$ is nilpotent, then its indeterminacy point corresponds to a Type III flag and all other points correspond to Type II flags. Thus, Theorem \ref{thm_main} reduces to the following semistability criterion. If $T$ is non-nilpotent, then up to $n/2 + q$ points may coincide at a general point of $\PP^1$, and up to $n/2$ points may coincide at each fixed point or indeterminate point of $T$. If $T$ is nilpotent, then up to $n/2 + q$ points may coincide at a general point of $\PP^1$, and up to $n/2 - q$ points may coincide at the indeterminacy point of $T$. Thus fixed points and indeterminacy points behave somewhat like marked points.
\end{example}

\begin{example} \label{ex_id}
For any $N, n, q \in \N$, if $T$ is the identity map, then there are no Type II or III flags, and each linear subvariety of $\PP^N$ defines a Type I flag. Thus Theorem \ref{thm_main} specializes to Mumford's stability test \eqref{eq_mumford_template}.
\end{example}

\begin{example} \label{ex_large_q}
For fixed choices of $N, n\in \N$ in Theorem \ref{thm_main}, if $q$ is sufficiently large, then the Type II condition of \eqref{eq_main} is vacuously true, and the Type III condition of \eqref{eq_main} is vacuously false. Thus, if $q$ is sufficiently large, a marked linear map $(T, v)$ is stable relative to $\shO(q,1,\hdots,1)$ if and only if $T$ is non-nilpotent and $v$ meets the Mumford stability test \eqref{eq_mumford_template} for each $T$-invariant proper subspace $H \subsetneq k^{N+1}$.
\end{example}

These examples suggest the following interpretation of Theorem \ref{thm_main}: the more that the map $T$ moves around the marked points, the better for GIT stability; the more marked points that eventually enter the kernel of $T$, the worse.

Our work is a starting point for describing the structure of $\mc{M}_{1,n}^N$ in more detail. Mumford and Suominen showed that the stable locus for the conjugation action of $\SL_{N+1}$ on $\PE$ is empty \cite{MR0437531}. Marking points and using an appropriate sheaf resolves this problem.
\begin{cor} \label{cor_stable_nonempty}
    Let $n \geq 1, N \geq 1$, and $q \geq n$. The stable locus in $\Twip$, relative to the sheaf $\shL = \shO(q, 1, \hdots, 1)$ with its unique linearization, is nonempty.
\end{cor}

The moduli space $\mc{M}^1_d$ of unmarked degree $d \geq 2$ dynamical systems on $\PP^1$ is rational \cite{MR2741188}. A central open problem in algebraic dynamics is to determine the rationality of the analogous moduli spaces $\mc{M}_d^N$ of dynamical systems on $\PP^N$ with $N > 1$.  In previous work, we showed that, in the planar case ($N = 2$), if $n \geq 4$, the moduli space $\mc{M}_{1,n}^2$ is birational to $\PP^{2n}$, by working in explicit coordinates \cite[Theorem 3.2]{MR4637161}. This generalizes to any dimension:

\begin{thm} \label{thm_rational}
For any $n, N \in \N$ and any very ample invertible sheaf $\shL$, if the stable locus of $\Twip$ relative to $\shL$ is nonempty, the variety $\mc{M}_{1,n}^N(\shL)$ is rational of dimension $nN$.
\end{thm}

\begin{question} We suggest some directions for future research into these moduli spaces.
\begin{enumerate}
\item Does Theorem \ref{thm_main} generalize to other Lie types?
\item What is the isomorphism type of the moduli space $\mc{M}_{1,2}^1 (\shL)$ of 2-marked linear maps on $\PP^1$?
\item What is the variation of the isomorphism type of $\mc{M}_{1,n}^N (\shL)$ with $\shL$?
\item What is the dimension of the GIT boundary? What are its irreducible components, particularly when $N = 1$?
\item For which $n, N, q \in \N$ and sheaves $\shL$ are the stable and semistable loci equal?
\end{enumerate}
\end{question}

We now sketch the proof of our main theorem, then discuss motivation from three areas: algebraic dynamics, combinatorial algebraic geometry, and integrable systems.

\subsection{Sketch of the proof of Theorem \ref{thm_main}}
The proof of Theorem \ref{thm_main} is via polyhedral combinatorics. There is a standard method in GIT for checking stability of a point. To any maximal torus in $\SL_{N+1}$ and point $p \in \Twip$, we associate a convex polytope called the \emph{weight polytope} or \emph{state polytope}. A point $p$ is GIT stable with respect to a given maximal torus if and only if its weight polytope contains the origin. 

We find that the weight polytopes arising from the conjugation action of $\SL_{N+1}$ on $\PE$ are \emph{root polytopes of $A$ type} \cite{MR4310906}. The vertices of these polytopes are root vectors of the $A_N$ lattice. However, the weight polytopes that arise from the combined action of $\SL_{N+1}$ on $\Twip$ are rather complicated. To deal with this issue, we introduce a variant of the weight polytope called the \emph{corner polyhedron} (Definition \ref{def_corner}). The corner polyhedron is the Minkowski sum of the standard weight polytope and the nonnegative orthant. The idea is that, since the Weyl group of the $A_N$ lattice is the symmetric group $\mc{S}_N$, every 1-parameter subgroup in $\SL_{N+1}$ can be put in a descending form, and the corner polyhedron describes stability relative to descending 1-parameter subgroups. This allows us to frame complicated systems of inequalities, like those in \cite{MR4132597}, in the intuitive language of convex geometry.

We show that the action of $\SL_{N+1}$ on the factor $(\PP^N)^n$ has the effect of translating the corner polyhedron for the factor $\PE$. Thus we understand the corner polyhedra for the combined action by computing all facets of the corner polyhedra of type $A_N$. Each facet imposes a necessary condition for GIT stability.

Our results on the vertices and facets of corner polyhedra of type $A_N$ are of independent combinatorial interest; they are contained in Proposition \ref{prop_vertices} and \ref{prop_facets}. For similar results about root polytopes, see \cite{MR4310906}. We also introduce a combinatorial gadget called the \emph{matrix profile} (Definition \ref{def_profile}) which classifies corner polyhedra. 

Finally, in Section \ref{sect_proof}, we finish the proof. We show that the many necessary conditions for GIT stability imposed by the facets of corner polyhedra can in fact be checked on just Type I, II, and III flags.

\subsection{Motivation from dynamical moduli spaces} \label{sect_alg_dyn}

A \emph{dynamical system of degree $d$} on $\PP^N$ is a morphism $\PP^N \to \PP^N$ of degree $d$. Since conjugating a morphism by a projective transformation produces a morphism with the same dynamics, the \emph{moduli space $\mc{M}_d^N$ of degree $d$ dynamical systems on $\PP^N$} is defined as the space of morphisms $\PP^N \to \PP^N$ up to the action of $\SL_{N+1}$ by conjugation. These moduli spaces were constructed with GIT by Silverman \cite{MR1635900} in degree $d \geq 2$ and dimension $N = 1$, and by Levy and Petsche-Szpiro-Tepper in degree $d \geq 2$ and dimension $N > 1$ \cite{MR2741188, MR2567424}. The analogous moduli space in degree $d = 1$ is impossible to construct because the GIT stable locus is empty \cite{MR0437531}.

The moduli space $\mc{M}^N_d$ admits a GIT semistable compactification $\bar{\mc{M}}^N_d$ that includes some morphisms of lower degree and rational maps; these degenerate objects are represented in the GIT boundary of the quotient, that is, $\bar{\mc{M}}^N_d \smallsetminus \mc{M}^N_d$. One might speculate that these rational maps are somehow dynamically special, although we note that the natural iteration maps between these moduli spaces have indeterminacy \cite{MR1764925}.

Now we discuss the reasons for considering marked points. The $\SL_{N+1}$-action \eqref{eq_acn}, defining projective equivalence of marked linear maps, preserves dynamical relationships among the marked points. That is, if we write
$$(T', v'_1, \hdots, v'_n) = A \cdot (T, v_1, \hdots, v_n),$$
then $A$ takes $T$-orbits to $T'$-orbits. Further, the orbit relations of the marked points are preserved; for instance, if for some $i, j$ we have $T(v_i) = v_j$, then also $T'(v'_i) = v'_j$. Any set of relations of this kind defines a subvariety of $\mc{M}_{1,n}^N$.

These notions were formalized by Doyle-Silverman as the theory of \emph{portrait moduli spaces} \cite{MR4132597}. They define the \emph{moduli space $\mc{M}_{d}^N(\mc{P})$ of degree $d \geq 2$ dynamical systems on $\PP^N$ with portrait $\mc{P}$}. A \emph{portrait} is a weighted digraph that prescribes dynamical relations on the marked points. Our moduli space $\mc{M}_{1,n}^N$ is the special case of this construction where $d = 1$ and $\mc{P}$ has $n$ vertices and no edges.

Doyle-Silverman obtained a partial description of GIT stability in degree $d \geq 2$ \cite[Theorem 8.1]{MR4132597}; our Theorem \ref{thm_main} is a complete description in degree $d = 1$. In principle, our techniques would give an exact characterization of GIT stability in degree $d \geq 2$, but the weight polyhedra will be much more complicated.

\subsection{Motivation from combinatorial algebraic geometry} \label{sect_dynamics_on_flags}
Definition \ref{defn:flaginkN1} associates a Hessenberg function $\Hess_{T, \mc{H}}$ to each flag $\mc{H}$ of length $\gamma + 1$ and projective linear map $T$. This function is a nondecreasing self-map of $\{0, \hdots, \gamma + 1\}$. The Hessenberg function captures the dynamics induced on the flag $\mc{H}$ by $T$, in the following sense: lifting $T$ to an endomorphism of $k^{N+1}$, there is a semiconjugacy from $T$ to $\Hess_{T, \mc{H}}$ as set-maps, defined by
\begin{align*}
    &k^{N+1} \to \{0, \hdots, \gamma + 1\},\\
    &v \mapsto \min \{j \colon v \in H_j \}.
\end{align*}

Hessenberg functions of flags arise in combinatorial algebraic geometry. Flags in $k^{N+1}$ are parametrized by \emph{flag varieties}, and each flag variety is stratified into \emph{Hessenberg varieties} via the associated Hessenberg functions relative to a fixed endomorphism $T$. The dynamics of $T$ tend to be reflected in the structure of the corresponding Hessenberg varieties. There is a substantial literature exploring the geometry of Hessenberg varieties in the complete flag variety, especially their dimension, singularities, and cohomology \cite{MR1043857, MR4391520, MR2275912}. Our work uses the notion of Hessenberg functions for incomplete flags; the corresponding generalized Hessenberg varieties are the subject of the recent paper \cite{incomplete_hessenberg}.

Theorem \ref{thm_main} shows that, whereas one normally expects each flag in $\PP^N$ to contribute a necessary condition for GIT stability, for our moduli problem, stability can be checked only on a few Hessenberg strata. This is a significant reduction; whereas the complete flag variety on $\PP^N$ has dimension $\binom{N + 1}{2}$, the Hessenberg variety of complete Type II flags has dimension at most $N$, and the Hessenberg variety of complete Type III flags has dimension at most $0$.

\subsection{Motivation from integrable systems} \label{sect_integrable_systems}

The moduli space $\mc{M}^N_{1,n}$ of $n$-marked linear maps on $\PP^N$ is the domain of a wide class of integrable systems called \emph{generalized pentagram maps} \cite{MR3253683, izosimov, MR3118623, MR3093293, MR2679816, MR2434454, MR3161305}. In this context, the space $\mc{M}^N_{1,n}$ is usually studied as a real or complex manifold that parametrizes \emph{twisted polygons}. 

\begin{defn} \label{defn_tw_poly}
Let $N, n \geq 1$ be integers. A \emph{twisted $n$-gon} $v$ in $\PP^N$ is a bi-infinite sequence $(v_i)_{i \in \Z}$  of points in $\PP^N$, called \emph{vertices}, such that there exists a unique invertible projective transformation $T$, the \emph{monodromy}, with the property that, for all $i \in \Z$,
$$Tv_i = v_{i + n}.$$
\end{defn}

A generic twisted $n$-gon is determined by its first $n$ vertices $v_1, \hdots, v_n$ and the monodromy $T$. This embeds the set of sufficiently generic twisted $n$-gons in the variety $\Twip$. Then the $\SL_{N+1}$-action \eqref{eq_acn} describes the effect of changing of coordinates on $\PP^N$, so the moduli space $\mc{M}^N_{1,n}$ parametrizes projective equivalence classes of twisted $n$-gons in $\PP^N$ and some degenerations.

Theorem \ref{thm_main} shows that the moduli space $\mc{M}^N_{1,n}$ can be equipped with the structure of an algebraic variety over any algebraically closed field, and provides a compactification $\bar{\mc{M}}^N_{1,n}$ in the category of varieties.

The moduli space of twisted $n$-gons is of interest because any projectively natural operation induces a rational self-map of $\mc{M}^N_{1,n}$ with an invariant fibration. For instance, the pentagram map \cite{MR1181089, MR2434454} is a dynamical system $\mc{M}^2_{1,n} \dashrightarrow \mc{M}^2_{1,n}$. Fixing $n \geq 4$ and cyclically ordering the vertices, we send each sufficiently generic twisted polygon $(v_i)$ to $(v'_i)$, where $$v'_i = \overline{v_{i - 2} v_{i + 1}} \cap \overline{v_{i - 1} v_{i + 2}}.$$
The map is projectively natural, so it defines a rational self-map of $\mc{M}^2_{1,n}$. Note that the $\SL_3$-equivalence class of the monodromy $T$ is an invariant of the pentagram map, providing two algebraically independent conserved quantities. In fact, the pentagram map is a discrete algebraic completely integrable system: it has an iterate which is birational to a translation on a family of abelian varieties \cite{MR3161305, MR4637161}.

Other maps on $\mc{M}^N_{1,n}$ that are known or believed to be integrable include generalized pentagram maps \cite{izosimov, MR3118623, MR3356734}, cross-ratio dynamics \cite{crossratdynamics}, and maps associated to moves on triple-crossing diagrams \cite{affolter2021crossratio}. More generally, the following recipe produces discrete dynamical systems on $\mc{M}^N_{1,n}$. Choose an integer $n' \leq n$, and let $\phi \colon (\PP^N)^{n'} \dashrightarrow \PP^N$ be any rational map which respects projective transformations of $\PP^N$. Then we get a (partially defined) self-map of the set of sequences $\Z \to \PP^N$, defined by the rule $(v_i) \mapsto (v'_i)$, where
$$v'_i = \phi(v_{i}, \hdots, v_{i+n'-1}).$$
Again there is an induced map $\Phi$ on $\mc{M}^N_{1,n}$ that admits at least $N$ independent integrals. Silverman, in a talk at the 2021 Joint Mathematics Meetings, proposed studying the integrability of $\Phi$ when $\phi$ is chosen to be the Cayley-Bacharach map $(\PP^2)^8 \dashrightarrow \PP^2$.

Thus we have the following curious observation: the moduli space $\mc{M}^N_{1,n}$ both \emph{is} a moduli space of dynamical systems with marked points, and itself admits many interesting rational dynamical systems.

\subsection{Outline}
Section \ref{sec_prelim} contains standard preliminaries from geometric invariant theory. New results begin in Section \ref{sect_polyhedra} with the definition of the corner polyhedron, the classification of the corner polyhedra of type $A_N$, and a menagerie of examples. In Section \ref{sect_proof}, we show that, instead of checking stability on every flag, one can just check Type I, II, and III flags, and we prove Theorem \ref{thm_main}, Corollary \ref{cor_stable_nonempty}, and Theorem \ref{thm_rational}. Section \ref{sect_one_marked_point} illustrates the main results in the special case $n = 1$ of one marked point.

\subsection*{Acknowledgments}
The author thanks his advisor, Joseph Silverman, for many thought-provoking discussions and for a careful reading of this manuscript. Further thanks to Niklas Affolter, Madeline Brandt, Brendan Hassett, Han-Bom Moon, Rafael Saavedra, Richard Schwartz, and Mariel Supina for helpful conversations. We also thank the anonymous referee. The author was supported by an NSF Graduate Research Fellowship under Grant No. 2040433. An earlier version of this work appeared 
as Chapter 4 of the author's Ph.D. thesis \cite{weinreich_thesis}.

\section{Preliminaries} \label{sec_prelim}
In this section, we recall some standard concepts from geometric invariant theory: categorical quotients, geometric quotients, good quotients, $G$-linearized sheaves, GIT stability, Hilbert-Mumford numerical invariants, destabilizing 1-parameter subgroups, weight sets, and the weight polytope. For a detailed development, see \cite[Chapter 6--9]{MR2004511}.

\begin{defn}
Let an algebraic group $G$ act algebraically on a variety $V$, and let $\alpha: G \times V \to V$ denote the action. A \emph{categorical quotient} is a pair $(V', \chi)$ consisting of a variety $V'$ and a $G$-invariant morphism $V \to V'$, such that for each variety $V''$, every $G$-invariant morphism $V \to V''$ factors through $\chi$ uniquely. We typically suppress the map $\chi$ from the notation.
A \emph{geometric quotient} is a categorical quotient $(V', \chi)$ with the additional property that if $v, v' \in V$ have the same image in $V'$, then they share a $G$-orbit.
A categorical or geometric quotient $(V', \chi)$ is called \emph{good} if it satisfies the additional properties:
\begin{enumerate}
    \item For any open subset of $U$ of $V'$, the homomorphism $\chi^*: \shO(U) \to \shO(\chi^{-1}(U))$ is an isomorphism onto the subring $\shO(\chi^{-1}(U))^G$ of $G$-invariant sections.
    \item If $W$ is a closed $G$-invariant subset of $V$, then $\chi(W)$ is a closed subset of $V'$.
    \item If $W_1, W_2$ are closed $G$-invariant subsets of $V$ with empty intersection, then their images in $V'$ have empty intersection.
\end{enumerate}

Let $V$ be a projective variety acted on by a geometrically reductive group $G$. Let $\shL$ be an invertible sheaf on $V$, and equip $\shL$ with a $G$-linearization. If $\shL$ is very ample, this data is more or less equivalent to a choice of embedding of $V$ into a projective space so that $G$ acts via a linear group. An extension of the $G$-action on $V$ to a linear action on that projective space is called a \emph{$G$-linearization}. A \emph{$G$-linearized sheaf} is a very ample invertible sheaf together with a choice of $G$-linearization.
\end{defn}

\begin{prop} \label{prop_sl_linearizations} 
For any very ample invertible sheaf $\shL$ on an irreducible normal projective variety $V$ acted on by $\SL_{N+1}$, there exists a unique $\SL_{N+1}$-linearization of $\shL$.
\end{prop}

Proposition \ref{prop_sl_linearizations} follows from the tools of \cite[Chapter 7]{MR2004511} and is an exercise in that chapter.

\begin{remark}
Our work fits into this setup, since $V = \Twip$ is projective and $G = \SL_{N+1}$ is geometrically reductive over any algebraically closed field. The very ample invertible sheaves up to isomorphism are classified by tuples $(q, m_1, \hdots, m_n) \in \N^{n+1}$; they are of the form $\shO(q, m_1, \hdots, m_n)$. Thus these tuples also classify $\SL_{N+1}$-linearized sheaves on $\Twip$.
\end{remark}

\begin{defn} \label{defn_stab}
Let $\shL$ be a $G$-linearized sheaf on $V$. A point $v \in V$ is \emph{semistable} with respect to $\shL$ if, for some $m \in \N$, the tensor power $\shL^{\otimes m}$
admits a global section $s \in \Gamma(V, \shL^{\otimes m})^G$ such that the set
$$ V_s \colonequals \{ y \in V \colon s(y) \neq 0 \}$$
is affine and contains $v$.

A point $v \in V$ is \emph{stable} if it is semistable and the set $V_s$ may be chosen such that the stabilizer of $v$ in $G$ is finite and all orbits of $G$ in $V_s$ are closed. A point $v \in V$ is \emph{unstable} if it is not semistable.
\end{defn}

Denote the sets of stable, semistable, and unstable points with respect to $\shL$ (with its $G$-linearization) by
$$\Vstable(\shL),\; \Vsemistable(\shL), \;\Vunstable(\shL).$$
In general, any of these sets could be empty.

Note that there is some variation in the literature regarding these definitions. In our notation, ``not stable'' and ``unstable'' mean different things. One could also consider the closely related concept of polystability, but we do not pursue this. 

The following theorem shows that the semistable quotient is a compactification of the stable quotient.
\begin{thm} \label{thm_stabquo}
Let $V$ be a variety with an algebraic action of a geometrically reductive group $G$. Let $\shL$ be a $G$-linearized sheaf.
\begin{enumerate}
    \item 
There exists a categorical quotient $\Vsemistable(\shL) \GITQ G$, and it is quasi-projective. If $V$ is projective, then $\Vsemistable(\shL) \GITQ G$ is projective.
\item
There exists a geometric quotient $ \Vstable(\shL) / G$.
\item
The stable locus $\Vstable(\shL)$ is a Zariski open subset of the semistable locus $\Vsemistable(\shL)$, which is in turn an open subset of $V$. The geometric quotient $\Vstable(\shL) / G$ is an open subset of the categorical quotient $\Vsemistable(\shL) \GITQ G.$
\end{enumerate}
\end{thm}
For a proof, see \cite[Theorem 8.1 and Proposition 8.1]{MR2004511}.

Suppose that the $G$-linearized sheaf $\shL$ embeds $V$ in $\PP^D$.
Then $G$ acts via a linear representation in $\SL_{D+1}$. Let $\Gm$ denote the multiplicative group scheme over $k$. Let $\ell \colon \Gm \to \SL_{D+1}$ be a 1-parameter subgroup. Every 1-parameter subgroup of $\SL_{D+1}$ is diagonal in some basis of $H^0(V,\shL)$. Such a basis is called a \emph{diagonalizing basis for} $\ell$. Choose a diagonalizing basis for $\ell$. For some $\lambda_1, \hdots, \lambda_{D+1} \in \Z$, the matrix of $\ell(\tau)$ is
\[
\ell(\tau) = \begin{bmatrix}
 \tau^{\lambda_1} \\
 & \tau^{\lambda_2} \\
 & & \ddots \\
 & & & \tau^{\lambda_{D+1}}
\end{bmatrix}.
\]
Suppose that $v \in V$, and in the diagonalizing basis for $\ell$, the homogeneous coordinates of $v$ are $v = [X_1 : \hdots : X_{D+1}]$. The \emph{numerical invariant} of $\ell$ at $v$ (with respect to $\shL$ with the chosen linearization) is
$$\mu(v, \ell) \colonequals \min_{1 \leq i \leq D+1} \{\lambda_i \colon X_i \neq 0 \}.$$
This quantity is independent of choice of diagonalizing basis for $\ell$. Note that the sign on $\mu$ varies in the literature.

Now we can state the \emph{Hilbert-Mumford numerical criterion}.

\begin{thm}[Mumford] \label{thm_numcrit}
With notation as above,
$$\Vsemistable = \{v \colon \mu(v, \ell) \leq 0 \text{ for all } \ell\},$$
$$\Vstable = \{v \colon \mu(v, \ell) < 0 \text{ for all } \ell \}.$$
\end{thm}

For a proof, see \cite[Theorem 9.1]{MR2004511}.

A \emph{de-semistabilizing 1-parameter subgroup} for $v \in V$ is a 1-parameter subgroup $\ell$ for which $\mu(v, \ell) \not\leq 0$.
A \emph{destabilizing 1-parameter subgroup} for $v \in V$ is a one-parameter subgroup $\ell$ for which $\mu(v, \ell) \not< 0$. By Theorem \ref{thm_numcrit}, semistability of $v \in V$ is equivalent to having no de-semistabilizing 1-parameter subgroups. Similarly, stability is equivalent to having no destabilizing 1-parameter subgroups.

We set more notation:
\begin{itemize}
    \item $\Theta$, a maximal torus in $\SL_{N+1}$;
    \item $W_\Z$, the character lattice of $\Theta$;
    \item $W$, the $\R$-vector space $W_\Z \otimes \R$;
    \item $R_\Z$, the lattice of 1-parameter subgroups of $\Theta$;
    \item $R$, the $\R$-vector space $R_\Z \otimes \R$;
    \item $\langle \cdot, \cdot \rangle: W \times R \to \R$ is the pairing defined as follows. For any $\chi \in W_\Z$ and $\ell \in R_\Z$, let $\langle \chi, \ell \rangle$ be the integer $\lambda$ such that the composed map $\chi \circ \ell : \Gm \to \Gm$ is $\tau \mapsto \tau^\lambda$. Extend linearly.
\end{itemize}
A choice of isomorphism $\Theta \cong \Gm^N$ induces an identification $W_\Z \cong \Z^{N+1}$. We also get an embedding $r \colon R_\Z \hookrightarrow \Z^{N+1}$ onto the hyperplane in $\Z^{N+1}$ where the coordinates sum to 0. Note that $R_\Z$ is isomorphic to the $A_N$ lattice rather than a free abelian group. With these identifications, the inner product $\langle \cdot, \cdot \rangle$ is computed by dot product in $\R^{N+1}$. 

The $\SL_{N+1}$-linearization of $\shL$ induces a representation $\Theta \hookrightarrow \GL_{D+1}$. Any such representation splits completely, because $\Theta$ is linearly reductive. Therefore we can choose a diagonalizing basis $X_1, \hdots, X_{D+1}$ for the action of $\Theta$ on $\shL$. Each basis element $X_i$ has an associated character $\chi_i \in W_\Z$. 

The \emph{weight} of $X_i$, denoted $\wt(X_i)$, is the image of the character $\chi_i$ in $\Z^{N+1}$. The \emph{weight set} $\wt(v) \subset W_\Z$ of $v \in V$, relative to the maximal torus $\Theta$ and the choice of identification $\Theta \cong \Gm^N$, is the set 
$$\wt(v) = \{ \wt(X_i) \colon X_i(v) \neq 0, \text{ for each } i = 1, \hdots, D + 1 \}.$$
The \emph{weight polytope} of $v$, denoted $\Pi(v)$, is the convex hull of $\wt(X_i)$ in $\R^{N+1}$. It is a convex polytope. We denote its interior by $\Pi(v)^\circ$.

We further define
$$\Vsemistable_\Theta = \{v \colon \mu(v, \ell) \leq 0 \text{ for all } \ell \subseteq \Theta \},$$
$$\Vstable_\Theta = \{v \colon \mu(v, \ell) < 0 \text{ for all } \ell \subseteq \Theta \}.$$

The weight polytope of a point $v$ depends on the identification $\Theta \cong \Gm^N$, but the weight set viewed in $W$ is basis-independent, so the statements $0 \in \Pi(v)$ and $0 \in \Pi(v)^\circ$ are basis-independent. In fact, we have the following consequence of Proposition \ref{thm_numcrit}:
\begin{prop}
    Given a maximal torus $\Theta$ of $\SL_{N+1}$ and an identification $\Theta \cong \Gm^N$, and a point $v \in V$, the following are equivalent:
    \begin{enumerate}
    \item $v \in V\gitss_\Theta$,
    \item $0 \in \Pi(v)$ for some diagonalizing basis of $\Theta$.
    \end{enumerate}
    Similarly, the following are equivalent:
    \begin{enumerate}
    \setcounter{enumi}{2}
        \item $v \in V\gits_\Theta$,
        \item $0 \in \Pi(v)^\circ$ for some diagonalizing basis of $\Theta$.
    \end{enumerate}
\end{prop}

For a proof, see \cite[Theorem 9.2]{MR2004511}.

Suppose that $\SL_{N+1}$ acts on a projective variety $V$ equipped with sheaf $\shL$. Let $B$ be a basis of $k^{N+1}$. The basis $B$ has an associated maximal torus $\Theta$ in $\SL_{N+1}$ containing all the 1-parameter subgroups of $\SL_{N+1}$ that are diagonal in $B$. A basis $B'$ for $H^0(V,\shL)$ is called \emph{$B$-aligned} if the representation of $\Theta$ in $\GL_{D+1}$ is diagonal in $B'$. (We do not require the weights to be in descending order.)

The next two lemmas are standard; they describe how weight polytopes relate to the Segre and Veronese maps.

\begin{defn}
Given two subsets $\Pi_1$, $\Pi_2$ of a $k$-vector space, the \emph{Minkowski sum} of $\Pi_1$ and $\Pi_2$ is the set
$$\Pi_1 + \Pi_2 \colonequals \{ \pi_1 + \pi_2 \colon \pi_1 \in \Pi_1, \, \pi_2 \in \Pi_2 \}.$$
Given any scalar $m \in k$ and subset $\Pi$ of a $k$-vector space, define
$$m\Pi \colonequals \{ m \pi \colon \pi \in \Pi \}.$$
\end{defn}

\begin{lemma} \label{lem_segre}
Suppose that $\SL_{N+1}$ acts on two projective varieties $V_1, V_2$ via the linearized sheaves $\shL_1, \shL_2$. 
We consider the product action of $\SL_{N+1}$ on $V_1 \times V_2$.
Suppose that $B$ is a basis of $k^{N+1}$ and $B_1, B_2$ are $B$-aligned bases of $H^0(V,\shL_1), H^0(V,\shL_2)$. There is a $B$-aligned basis $B_1 \times B_2$ formed of all products of the form $b_1 b_2$, where $b_1 \in B_1$ and $b_2 \in B_2$.

The weight polytope of $(p_1, p_2)$ with respect to $B_1 \times B_2$ is the Minkowski sum
$$\Pi(p_1, p_2) = \Pi(p_1) + \Pi(p_2).$$
\end{lemma}

\begin{lemma} \label{lem_veronese}
Suppose that $\SL_{N+1}$ acts on a projective variety $\mc{X}$ via the sheaf $\shL$. For any $m \in \N$, we consider the induced action of $\SL_{N+1}$ on $\mc{X}$ with the tensor power sheaf $\shL^{\otimes m}$. Let $B$ be a basis of $k^{N+1}$, and Suppose that $B'$ is a $B$-aligned basis of $H^0(V,\shL)$. There is a $B$-aligned basis $B''$ of $H^0(V,\shL^{\otimes m})$ consisting of all $m$-fold products of elements of $B'$.

Let $\Pi(p)$ denote the weight polytope of $p$ for $\shL$, and let $\Pi^{(m)}(p)$ denote the weight polytope of $p$ for $\shL^{\otimes m}$. Then
$$\Pi^{(m)}(p) = m\Pi(p).$$
\end{lemma}

The proofs are direct applications of the definitions of the Segre and Veronese embeddings; we omit them.

\section{Corner polyhedra and stability} \label{sect_polyhedra}

In this section, we introduce our key technical tool, the corner polyhedron. We use corner polyhedra instead of weight polytopes to simplify the analysis of GIT stability. Our first step is to derive the stability-detecting property of corner polyhedra from the corresponding property of weight polytopes. The corner polyhedron stability test is given in Proposition \ref{prop_corner_numerical}.

We recall some terminology. The setting is any $\R$-vector space. Unless stated otherwise, half-spaces are assumed to be closed.
\begin{itemize}
    \item A \emph{convex polyhedron} is an intersection of finitely many half-spaces. All polyhedra we consider are convex, so we may just write \emph{polyhedron}.
    \item A \emph{convex polyhedral cone} is an intersection of finitely many half-spaces, all of which contain the origin at the boundary.
    \item  A \emph{polytope} is a bounded polyhedron.
    \item A \emph{supporting halfspace} of a polyhedron is a halfspace containing the polyhedron and containing some point on its boundary. 
    \item A \emph{face} of a polyhedron is any set which is the intersection of that polyhedron with a supporting halfspace.
    \item For any $D \in \N$, a \emph{facet} of a $D$-dimensional polyhedron is a face of dimension $D - 1$. 
    \item Each facet has a unique \emph{associated supporting halfspace}. 
    \item A \emph{vertex} is a 0-dimensional face.
    \item A polyhedron is \emph{full-dimensional} if it is a polyhedron with topological dimension equal to the ambient vector space. 
    \item A polyhedron is \emph{pointed} if it has at least one vertex.
    \item A polytope is \emph{integral} with respect to a lattice if its vertices are lattice points.
    \item A polyhedron is \emph{rational} with respect to a chosen basis if the associated supporting halfspace of each facet is defined over $\Q$ when written in that basis.
\end{itemize}

We use the following basic principles of convex polyhedra throughout. For any set $\zeta$ in an $\R$-vector space, let $\Conv(\zeta)$ denote its convex hull. We use $+$ to denote Minkowski sum.

\begin{itemize}
    \item The convex hull of a finite set is a polytope.
    \item A convex polyhedron is equal to the intersection of the associated supporting halfspaces of its facets.
    \item Given convex polyhedra $\Pi_1, \Pi_2$, the Minkowski sum $\Pi_1 + \Pi_2$ is a polyhedron. Each face of $\Pi_1 + \Pi_2$ is the Minkowski sum of some face of $\Pi_1$ and some face of $\Pi_2$.
    \item Every convex polyhedron can be expressed as the Minkowski sum of a convex polyhedral cone and a convex polytope.
    \item Any functional on a polytope achieves its minimum at some vertex.
    \item An integral polytope is also rational as a polyhedron.
    \item The Minkowski sum of rational polyhedra is rational.
\end{itemize}

For any $N \in \N$, let $\R^N_+$ denote the \emph{nonnegative orthant}, i.e., the set of points in $\R^N$ with all nonnegative coordinates.
The \emph{corner hull} of a set $\zeta \subseteq \R^N$, denoted $\Corner(\zeta)$, is defined by
$$\Corner(\zeta) = \Conv(\zeta) + \R^N_+,$$
where $\R^N_+$ denotes the nonnegative orthant.

We repeatedly use the following properties of the corner hull:
\begin{enumerate}
    \item The corner hull of a finite set is a finite, full-dimensional, pointed, unbounded, convex polyhedron. The vertices of the corner hull of a set $\zeta$ are themselves elements of $\zeta$.
    \item For any sets $\zeta, \zeta' \subset \R^N$,
    $$\Corner(\zeta + \zeta') = \Corner(\zeta) + \Corner(\zeta').$$
    \item For any real number $m > 0$, and any set $\zeta \subset \R^N$,
    $$\Corner(m \zeta) = m \Corner (\zeta).$$
\end{enumerate}

Our next lemma is a version of the hyperplane separation theorem, a central tool in convex geometry, for the corner hull. It shows that membership in the corner hull can be tested by functionals with positive coordinates in the dual space.

\begin{lemma}[Positive hyperplane separation] \label{lem_hyperplane_sep}
Let $(\R^N)^\vee_+$ denote the nonnegative orthant in the dual space to $\R^N$, with respect to the dual basis. Let $(\Z^N)^\vee_+$ denote the set of standard lattice points within it.
\begin{enumerate}
    \item 
For any finite set $\zeta \subset \R^N$,
we have $0 \in \Corner( \zeta )$
if and only if there is no $r \in (\R^N)^\vee_+ \smallsetminus \{0\}$ such that, for all $w \in \zeta$,
$$r(w) > 0.$$
\item
For any finite set of lattice points $\zeta \subset \Z^N$,
we have $0 \in \Corner( \zeta )$
if and only if there is no lattice point $r \in (\Z^N)^\vee_+ \smallsetminus \{0\}$ such that, for all $w \in \zeta$,
$$r(w) > 0.$$
\item
For any finite set $\zeta \subset \R^N$, we have 
$0 \in \Corner( \zeta )^\circ$
if and only if there is no $r \in (\R^N)^\vee_+ \smallsetminus \{0\}$ such that, for all $w \in \zeta$,
$$r(w) \geq 0.$$
\item
For any finite set of lattice points $\zeta \subset \Z^N$, we have 
$0 \in \Corner( \zeta )^\circ$
if and only if there is no lattice point $r \in (\Z^N)^\vee_+ \smallsetminus \{0\}$ such that, for all $w \in \zeta$,
$$r(w) \geq 0.$$
\end{enumerate}
\end{lemma}

\begin{proof}
 \par(1)\enspace Suppose that $0 \in \Corner(\zeta)$. Then there exists some point $w_0 \in \Conv(\zeta)$ such that, for all $r \in (\R^N)^\vee_+,$
 $$0 \geq r(w_0).$$
 Functionals on $\Conv(\zeta)$ achieve their minima on an element of $\zeta$. So, for each $r \in (\R^N)^\vee_+$, there exists some $w \in \zeta$, depending on $r$, such that $r(w_0) \geq r(w)$. Thus
 $$0 \geq r(w).$$
 For the other direction, suppose that $0 \not\in \Corner(\zeta)$. There exists some shortest-length vector $w_0$ in $\Corner(\zeta)$. By the hyperplane separation theorem, the functional $r$ dual to $w_0$ takes only positive values on $\Corner(\zeta)$. In particular we know that $r$ takes only positive values on $\zeta$. The coordinates of $w_0$ are non-negative, since otherwise the negative coordinates could be replaced by 0 to obtain a shorter-length vector in $\Corner(\zeta)$. The coordinates of $w_0$ are rational, since the points of $\zeta$ are lattice points. Thus some positive integer multiple of $r$ is in $(\R^N)^\vee_+ \smallsetminus \{0\}.$
 
 \par(2)\enspace
 One direction follows immediately from (1). For the other direction, note that in the proof of (1), since $\Corner(\zeta)$ is the sum of an integral polytope and a rational polyhedral cone, all its facets are rational. Thus the shortest vector $w_0$ has rational coordinates, so a sufficiently high multiple of its dual is the desired lattice point $r$.
 
 \par(3)\enspace  Let $O^+$ be the nonnegative orthant and let $(O^+)^\circ$ be its interior. We have
 $$\Corner(\zeta)^\circ = (O^+)^\circ + \Conv(\zeta).$$
 Thus if $0 \in \Corner(\zeta)^\circ$, then there exists some point $w_0 \in \Conv(\zeta)$ such that, for all $r \in (\R^N)_+^\vee \smallsetminus \{0\}$,
 $$0 > r(w_0).$$
 Following the same argument as in (1), for each $r$ there exists some $w \in \zeta$ such that $r(w_0) \geq r(w)$, so $0 > r(w)$. 
 
 For the other direction, we already proved it in (1), except for the case that $0$ is on the boundary of $\Corner(\zeta)$. Then $0$ belongs to a facet of $\Corner(\zeta)$. Take $r$ to be the functional defining that facet's associated supporting halfspace.
 
 \par(4)\enspace Similar to (2) and (3); omitted.
\end{proof}

It turns out that, because we are working with $\SL_{N+1}$-actions, computing in the space $\R^{N+1}$ is redundant. All our computations happen in the hyperplane of vectors with coordinates that sum to 0. The next definitions set up a preferred basis for this hyperplane and its dual space.

Let $N \in \N$. We consider the lattice $\Z^{N+1}$ embedded in $W = \R^{N+1}$. We denote the standard basis vectors by $e_1, \hdots, e_{N+1}$. We denote the dual space of $\R^{N+1}$ by $R = (\R^{N+1})^\vee$. The space $R$ is spanned by the dual coordinates to the standard basis vectors, denoted $x_1, \hdots, x_{N+1}$.

Let $\bm{1} = (1, \hdots, 1) \in W$. Let $\bar{W}$ be the hyperplane orthogonal to $\bm{1}$. The elements of $W$ that belong to $\bar{W}$ are vectors with coordinates that sum to 0.

Let $\bar{R} = (\bar{W})^\vee$. We consider the basis $s_1, \hdots, s_N$ of $\bar{R}$ defined by
$$s_i = \sum_{j=1}^i x_j.$$
Let $\bar{R}^+$ denote the nonnegative $\R$-span of $s_1, \hdots, s_N.$

We work with this particular basis $s_1, \hdots, s_N$ because it has the following convenient property. There is a restriction map $R \to \bar{R}$ defined by restricting functionals from $W$ to $\bar{W}$. For any $r \in R$, we denote its restriction by $\bar{r}$. If $r = (r_1, \hdots, r_{N+1})$, we have $r_1 \geq \hdots \geq r_{D+1}$ if and only if $\bar{r} \in \bar{R}^+$. Further, if $r(\bm{1}) = 0$, then we have $\bar{r} = 0$ if and only if $r_1 = \hdots = r_{D+1} = 0.$

Let $\Orth$ denote the nonnegative orthant in $\bar{W}$ relative to the dual basis $s_1, \hdots, s_N$. Its interior $\OrthInt$ is the strictly positive orthant.

The \emph{corner hull} of a set $\zeta \subseteq \bar{W}$, denoted $\Corner(\zeta)$, is defined as the corner hull of the image of $\zeta$ in $\R^N$ via the identification $\bar{W} \cong \R^N$ from the preferred basis $s_1, \hdots, s_N$. That is,
$$\Corner(\zeta) = \Conv(\zeta) + \Orth.$$

\begin{defn} \label{def_corner_poly}
Consider a projective variety $V$ equipped with a $\SL_{N+1}$-linearized sheaf. Let $B$ be a basis of $k^{N+1}$. Let $W \cong \R^{N+1}$ be the vector space over the character lattice corresponding to basis $B$, and let $\bar{W}$ be the hyperplane of points of $Z$ with coordinates that sum to 0. Let $\Pi(v) \subset W$ denote the weight polytope of $v$ with respect to this basis. Let $\bar{\Pi}(v)$ denote the projection of $\Pi(v)$ to $\bar{W}$. The \emph{corner polyhedron} of $v \in V$, denoted $\Pi_+(v)$, is the corner hull 
$$\Corner(\bar{\Pi}(v)) \subseteq \R^N.$$
We call the vertices of $\Pi_+(v)$ the \emph{control points}. The set of control points is denoted $\Ctrl(v)$.
\end{defn}

We emphasize that the weight polytope $\Pi(v)$ is a subset of the vector space $W$, which has dimension $N + 1$, whereas the corner polyhedron $\Pi_+(v)$ is a subset of the subspace $\bar{W}$, which has dimension $N$.

Let $B$ be a basis, with vectors $b_1, \hdots, b_{N+1}$. We emphasize that $B$ is ordered. A 1-parameter subgroup of $\SL_{N+1}$ is said to be \emph{normalized along $B$} if, written in basis $B$, it is of the form $\diag [\tau^{r_1}, \tau^{r_2}, \hdots, \tau^{r_{N+1}}]$, where $r_1 \geq \hdots \geq r_{N+1}$. Every $1$-parameter subgroup of $\SL_{N+1}$ is normalized along some basis $B$ because we can permute basis vectors to get a new basis.

Let $V$ be a variety acted on by $\SL_{N+1}$, equipped with a linearized sheaf $\shL$. Given an ordered $B$, an element $v \in V$ is called \emph{$(B,\shL)$-stable} if there is no destabilizing 1-parameter subgroup of $\SL_{N+1}$ normalized along $B$. Similarly, a point $v$ is \emph{$(B,\shL)$-semistable} if there is no de-semistabilizing 1-parameter subgroup normalized along $B$.

Now we interpret the Hilbert-Mumford criterion in these terms. The idea to interpret the Hilbert-Mumford criterion with the weight polytope is standard; see \cite[Chapter 9]{MR2004511}. The original contribution in Proposition \ref{prop_corner_numerical} is that the corner hull takes advantage of the symmetry of $\SL_{N+1}$.

\begin{prop} \label{prop_corner_numerical}
Suppose $\SL_{N+1}$ acts on a projective variety $V$ via the sheaf $\shL$. Let $B$ be a basis of $k^{N+1}$. Then $v \in V$ is $(B,\shL)$-semistable if and only if
$$ 0 \in \Pi_+(v).$$
The point $v \in V$ is $(B,\shL)$-stable if and only if
$$ 0 \in \Pi_+(v)^\circ.$$
\end{prop}

\begin{proof}
 One checks $(B,\shL)$-semistability by computing the numerical invariant $\mu(v, \ell)$ for every 1-parameter subgroup $\ell$ normalized along $B$. We have
 $$\mu(v, \ell) = \min_{w \in \wt(v)} r(\ell)(w).$$

 For any $w \in W$, let $\bar{w}$ denote the projection of $w$ to $\bar{W}$. For any $\ell$, the vector $r(\ell)$ is orthogonal to $(1, \hdots, 1)$. So, for any $\ell$ and $w$, we have
 $$r(\ell)(w) = r(\ell)(\bar{w}) = \bar{r}(\ell)(\bar{w}).$$
 As $\ell$ ranges over $1$-parameter subgroups along $B$, the vector $\bar{r}(\ell)$ ranges over nonzero lattice points in $\bar{R}^+$.
 
 Thus $v$ is $(B,\shL)$-semistable if and only if, for all nonzero lattice points $\bar{r} \in \bar{R}^+$, there exists $w \in \wt(v)$ such that
 $$ \bar{r} ( \bar{w} ) \leq 0.$$
 By the choice of basis $s_1, \hdots, s_N$ on $\bar{R}$ and Lemma \ref{lem_hyperplane_sep} (2), this is equivalent to
 $$0 \in \Pi_+(v).$$
 The $(B,\shL)$-stability claim follows from the same argument, but applying Lemma \ref{lem_hyperplane_sep} (4) instead.
\end{proof}

\subsection{First application: configurations of points} \label{sec_first_application}

In this section, we show how to analyze the diagonal action of $\SL_{N+1}$ on $(\PP^N)^n$ with the corner polyhedron; here, no linear map $T$ is involved. These results are essential for our proof of Theorem \ref{thm_main}, and we recover a classical result of Mumford as Theorem \ref{thm_classic_git} along the way.

For any $m \in \N$, let $[m] \subset \N$ denote the set $\{1, \hdots, m\}$.

We first compute the corner polyhedron associated to a given basis for a single point in $\PP^N$; it is just a translate of the nonnegative orthant.

\begin{prop} \label{prop_one_point}
    Consider the action of $\SL_{N+1}$ on $\PP^N$, with $\shL = \shO(m)$ for some $m \geq 1$. Choose a basis $B$ of $k^{N+1}$, let $X_1, \hdots, X_{N+1}$ be induced coordinates on $\PP^N$, and let $B'$ be the induced $B$-aligned basis from the Segre map. Given a point $v \in \PP^N$, Let $i$ be the largest index such that $X_i(v) \neq 0$. Then
    $$\Ctrl(v, \shL) = \left\{ m \cdot \left( -\frac{1}{N+1}, \hdots, -\frac{i - 1}{N+1}, 1 - \frac{i}{N+1}, \hdots, 1 - \frac{N}{N+1} \right) \right\}.$$
\end{prop}

\begin{proof}
 Assume $m = 1$, since then the general case follows from Lemma \ref{lem_veronese}.
 
 Let $\ell$ be a 1-parameter subgroup along a basis $B$, with $r(\ell) = (r_1, \hdots, r_{N+1})$. Let $X_1, \hdots, X_{N+1}$ be generators of $\shO(1)$ obtained from $B$. Then, for all $i \in [N+1]$, the weight of $X_i$ is the standard basis vector
$$\wt(X_i) = e_i.$$
 Therefore, we have
 \[
 s_j(\wt(X_i)) =
 \begin{cases}
  0 &  \text{ if } j < i, \\
  1 & \text{ otherwise.} \\
 \end{cases}
 \]
 Since the sum of the entries of $\wt(X_i)$ is 1, we calculate
 $$\barwt(X_i) = \wt(X_i) - \frac{1}{N+1} \bm{1}.$$
 We have $s_j(\bm{1}) = j$. This gives us
  \[
 s_j(\barwt(X_i)) =
 \begin{cases}
  -\frac{j}{N+1} &  \text{ if } j < i, \\
  1 - \frac{j}{N+1} & \text{ otherwise.} \\
 \end{cases}
 \]
 Thus, for all $(i, i', j)$ such that $i < i'$, we have 
 $$s_j(\barwt(X_{i})) \leq s_j(\barwt(X_i')).$$
 It follows that every element of $\barwt(v)$ belongs to the nonnegative orthant at $\barwt(X_i)$, where $i$ is the largest index for which $X_i(v) \neq 0$.
\end{proof}

Let $\bm{m} = (m_1, \hdots, m_n)$ be an $n$-tuple of positive integer weights. Each such vector determines a very ample sheaf $\shO(\bm{m})$ on $(\PP^N)^n$. For present and later use, we define a vector
\begin{equation} \label{eq_config_vector}
    \eta(v) := \left(
    \sum_{i=1}^n \frac{-j m_i}{N+1} + \sum_{i \in [n]: \; X_{j+1}(v_i) = \hdots = X_{N+1}(v_i) = 0} m_i
    \right)_{j \in [N]}.
\end{equation}

We now extend the previous proposition to ordered configurations of points in $\PP^N$.

\begin{prop} \label{prop_tuple}
    Consider the diagonal action of $\SL_{N+1}$ on $(\PP^N)^n$, with $\shL = \shO(\bm{m})$. Choose a basis $B$ of $k^{N+1}$, and let $B'$ be the induced $B$-aligned basis from the Segre map. Then, for each $v = (v_1, \hdots, v_n) \in (\PP^N)^n.$
    \[
    \Ctrl(v, \shL) = \{ \eta(v) \}.
    \]
\end{prop}

\begin{proof}
The case $n = 1$ is the result of Proposition \ref{prop_one_point}. By Lemma \ref{lem_segre}, the set $\Ctrl(v, \shL)$ consists of just one vector, and we can compute it as the sum, for $1 \leq i \leq n$, of the control vector of $v$ relative to $\shO(m_i)$. The result follows.
\end{proof}

As a corollary, we obtain Mumford's classical description of GIT stability of projective configurations, which serves as the template for Theorem \ref{thm_main}. We note that the proof we give here is essentially the proof that can be found in \cite[Chapter 11]{MR2004511}, translated into the language of convex geometry.

\begin{thm}[Mumford, \cite{MR1304906}] \label{thm_classic_git}
    Consider the diagonal action of $\SL_{N+1}$ on $(\PP^N)^n$, with $\shL = \shO(\bm{m})$. Then $v = (v_1, \hdots, v_n)$ is semistable if and only if, for every nonzero proper subspace $H \subset k^{N+1}$, we have
    \begin{equation} \label{eq_mum_classic}
    \sum_{i \in [n] \colon v_i \in H} m_i \leq \frac{\dim H}{N + 1} \left( \sum_{i=1}^n m_i \right).
    \end{equation}
    The configuration $v$ is stable if and only if the inequality \eqref{eq_mum_classic} is strict.
\end{thm}

\begin{proof}
By Proposition \ref{prop_corner_numerical}, the point $v$ is semistable if and only if, in each basis $B$ of $k^{N+1}$, there is a $B$-aligned basis $B'$ of $\shL$ such that $0 \in \Pi_+(v)$. For any given basis $B$, we let $B'$ be the natural one coming from the Segre and Veronese embeddings (Proposition \ref{prop_tuple}). By Proposition \ref{prop_tuple}, a point $v$ is $(B,\shL)$-semistable if and only if each entry of $\eta(v)$ is nonpositive. So, for each $j \in [N]$, the $j$-th coordinate of of $\eta(v)$ imposes the condition
\[
     \sum_{i \in [n]: \; X_{j+1}(v_i) = \hdots = X_{N+1}(v_i) = 0} m_i \leq \sum_{i=1}^n \frac{j m_i}{N+1}.
\]
Let $H_j = \Span(X_1, \hdots, X_j)$, viewed as a subspace of $k^{N+1}$. Then $j = \dim H_j$, and the condition $X_{j+1}(v_i) = \hdots = X_{N+1}(v_i) = 0$ is equivalent to $v \in H_j$. Therefore, $(B,\shL)$-semistability is equivalent to satisfying \eqref{eq_mum_classic} for the sequence of subspaces $H_1, \hdots, H_N$ coming from basis $B$. Every nontrivial proper subspace appears this way for some basis $B$, so the result follows.

For stability, the same argument works using strict inequalities.
\end{proof}
 
\subsection{Corner polyhedra of type \texorpdfstring{$A_N$}{A N}} Having understood stability of configurations of marked points, we turn to the complementary problem: describing corner polyhedra of unmarked linear maps.

The main result of this subsection, Proposition \ref{prop_vertices}, calculates the vertices of the corner polyhedron associated to any matrix $M$. We use this description to classify corner polyhedra (Corollary \ref{cor_bijection}) in terms of a combinatorial gadget called the \emph{profile} of a matrix (Definition \ref{def_profile}). We encourage the reader to consult Section \ref{sect_examples} for examples to see the general theory in action.

Consider the action of $\SL_{N+1}$ by conjugation on $\PE$. Let $B$ be some basis of $k^{N+1}$. We get an induced basis of the sheaf $\shO(1)$ on $\PE$ by considering the $(N + 1)^2$ matrix coordinates in basis $B$; we denote these coordinates by $a_{ij}$, where $i, j \in [N+1].$ Any 1-parameter subgroup $\ell$ that is diagonal in $B$ has an associated vector $r(\ell) = (r_1, \hdots, r_{N+1})$, and we have
$\ell(\tau) \cdot a_{ij} = \tau^{r_i - r_j}$. It follows that, with the choice of sheaf $\shL = \shO(1)$,
$$\wt(a_{ij}) = e_i - e_j.$$

Notice that these weights $\wt(a_{ij})$ are roots of the $A_N$ lattice.

\begin{defn}
A \emph{root polytope of type $A_N$} is the convex hull of a nonempty set of vectors of the form $e_i - e_j$, where $i, j \in [N+1]$. Note that we allow $i = j$. The \emph{full root polytope of type $A_N$} is the root polytope defined with all possible vectors $e_i - e_j$.
\end{defn}
The combinatorics of root polytopes is known in detail; see \cite{MR2801233, MR1697418} and especially \cite{MR4310906}. Full root polytopes are orbit polytopes, since their vertices are the orbit of $(1,-1,0,\hdots,0)$ by the symmetric group $S_{N+1}$.

We introduce the analogous definition for corner polyhedra.
\begin{defn} \label{def_corner}
A \emph{corner polyhedron of type $A_N$} is the corner hull $\Corner(\zeta)$ of a finite, nonempty set $\zeta$ of vectors of the form $e_i - e_j$, where $i, j \in [N+1]$. Note that we allow $i = j.$
\end{defn}

By Definition \ref{def_corner_poly}, the corner polyhedra of type $A_N$ are exactly the polyhedra of the form $\Pi_+(M)$, for some $M \in \PP(\Mat_{N+1})$.

To compute the vertices and faces of corner polyhedra, it is convenient to work in the hyperplane of vectors in $\R^{N+1}$ with coordinates that sum to 0, for which we have a preferred basis. We recall our notation:
\begin{itemize}
    \item $W$, the $\R$-vector space over the character lattice in basis $B$, identified with $\R^{N+1}$ by the basis $e_1, \hdots, e_{N+1}$.
    \item $\bar{W} \subset W$, the hyperplane orthogonal to $\bm{1}$.
    \item $R$, the dual space to $W$, identified with $\R^{N+1}$ by the basis $x_1, \hdots, x_{N+1}$.
    \item $s_i$, for each $i \in [N]$, defined as the functional $x_1 + \hdots + x_i$.
    \item For any $w \in W$, we write $\bar{w}$ for the projection of $w$ to $\bar{W}$, identified with $\R^N$ by the dual basis to $s_1, \hdots, s_N$.
    \item $\bar{R}$, the dual space of $\bar{W}$, identified with $\R^N$ by the basis $s_1, \hdots, s_N \in \bar{R}$.
    \item $\Pi_+(M)$, the corner polyhedron associated to $M$.
\end{itemize}

Our next lemma calculates weights in the basis $s_1,\ldots,s_n$. Note that the trichotomy that appears here (\emph{on} vs. \emph{below} vs. \emph{above} the main diagonal) persists through the rest of our argument, and ultimately leads to the definitions of Type I, II, and III flags in Theorem \ref{thm_main}.

\begin{lemma} \label{lem_root_weights}
Let $a_{ij}$ be an entry coordinate of $\Mat_{N+1}$. Then $\barwt(a_{ij})$ depends on the $a_{ij}$ location relative to the main diagonal, as follows:
\begin{enumerate}
    \item If $i = j$, then $\barwt(a_{ij}) = 0$.
    \item If $i > j$, then
    $$\barwt(a_{ij}) = (0, \hdots, 0, -1, \hdots, -1, 0, \hdots, 0),$$
    where the $-1$-block runs from $j$ to $i - 1$ inclusive.
    \item If $i < j$, then
    $$\barwt(a_{ij}) = (0, \hdots, 0, 1, \hdots, 1, 0, \hdots, 0),$$ where the 1-block runs from $i$ to $j - 1$ inclusive.
\end{enumerate}
\end{lemma}
\begin{proof}
First, notice that $\wt(a_{ij}) = e_i - e_j$ is in $\bar{W}$ already, so $\barwt(a_{ij}) = \wt(a_{ij})$. The result follows from evaluating the functionals $s_1, \hdots, s_N$ on $\barwt{(a_{ij})}$.
\end{proof}

We now discuss the simplest examples in dimension $N = 1, 2, 3$. We include a catalogue of more elaborate examples in Section \ref{sect_examples}.
\begin{example}
Let $M \in \Mat_{N+1}$ be a matrix with all entries nonzero. Then the weight polytope $\Pi(M)$ is the full root polytope of type $A_N$.
\begin{itemize}
    \item If $N = 1$, then $\Pi(M)$, viewed as a subset of $W = \R^2$, is the line segment from $(-1,1)$ to $(1,-1)$. In the preferred basis $s_1$ on $\bar{W}$, the polytope $\Pi(M)$ is the interval $[-1, 1]$. The corner polyhedron $\Pi_+(M)$ is the ray $[-1, \infty)$.
    \item If $N = 2$, then $\Pi(M)$ is a regular hexagon in $W = \R^3$. Its vertices in $W$ are
    \[
    (-1,0,1), (-1,1,0), (0,-1,1), (1,0,-1), (1, -1, 0), (0,1,-1).
    \]
    Note that $(0,0,0) \in \wt(M)$ is not a vertex of $\Pi(M)$ because it is in the interior. In the basis $s_1, s_2$ on $\bar{W}$, the vertices of $\Pi(M)$ are
\[
(-1,-1), (-1,0), (0,-1), (1,1), (1,0), (0,1).
\]
    We see from this description that the corner polyhedron $\Pi_+(M)$ is the nonnegative orthant translated by $(-1,-1)$.

\item If $N = 3$, then $\Pi(T)$ is a cuboctahedron. In $W$, its vertices are the 12 vectors with entries $0, 0, -1, 1$ in some order. In $\bar{W}$, its vertices are
\begin{align*}
&(-1,-1,-1), (-1,-1,0),(0,-1,-1), (-1,0,0), (0,-1,0), (0,0,-1), \\
&(1,1,1), (1,1,0),(0,1,1), (1,0,0), (0,1,0), (0,0,1).
\end{align*}
Thus $\Pi_+(M)$ is the nonnegative orthant translated by $(-1,-1,-1)$.
\end{itemize}
\end{example}

As the examples show, the corner polyhedron $\Pi_+(M)$ is generally much simpler than the weight polytope $\Pi(M)$.

We define a partial order on the coordinates of $\PE$ that will help us identify the vertices of the corner polyhedron.

\begin{defn} \label{def_outweighing}
Given $(i,j), (i',j') \in [N+1]^2$, we say $(i, j)$ \emph{outweighs} $(i', j')$ if $\barwt(a_{i'j'}) \in \barwt(a_{ij}) + \Orth$ and $\barwt(a_{i'j'}) \neq \barwt(a_{ij})$. Equivalently, entry $(i, j)$ outweighs $(i', j')$ if $\barwt(a_{i'j'}) \neq \barwt(a_{ij})$ and, for each functional $s = s_1, \hdots, s_N$, we have
$$s(\barwt(a_{ij})) \leq s(\barwt(a_{i'j'})).$$

The outweighing relation is a strict partial order on $[N+1]^2$. That is, outweighing is irreflexive, transitive and antisymmetric.

A \emph{minimal element} with respect to outweighing is an entry that is not outweighed by any other entries. To describe $\Pi_+(M)$, we only need to know the nonzero entries $(i,j)$ of $M$ that are minimal with respect to outweighing, among all nonzero entries of $M$.
\end{defn}

Outweighing can be described explicitly by cases, depending on the location of the entries $(i,j), (i',j')$ relative to the main diagonal.

\begin{lemma} \label{lem_outweighing}
Let $(i,j), (i',j') \in [N+1]^2.$
\begin{enumerate}
    \item If $i > j$ and $i' > j'$, then $(i, j)$ outweighs $(i', j')$ if and only if $(i, j) \neq (i',j')$ and $i \geq i'$ and $j \leq j'$.
    \item If $i > j$ and $i' \leq j'$, then $(i, j)$ outweighs $(i', j')$.
    \item If $i < j$ and $i' < j'$, then $(i,j)$ outweighs $(i', j')$ if and only if $(i, j) \neq (i', j')$ and $i \geq i'$ and $j \leq j'.$
\end{enumerate}
\end{lemma}
\begin{proof}
 These results follow immediately from Lemma \ref{lem_root_weights}.
\end{proof}

\begin{example}
Suppose that $N = 3$, so $M \in \Mat_{4}$.
\begin{itemize}
    \item The entry $(4,2)$ outweighs every entry on or above the main diagonal, and also outweighs $(4,3)$ and $(3,2)$. 
    \item The entry $(2,4)$ only outweighs $(3,4)$ and $(4,4)$.
    \item The entry $(2,2)$ outweighs every entry above the main diagonal.
\end{itemize}
\end{example}

To keep track of minimal elements (Definition \ref{def_outweighing}), we introduce a combinatorial gadget called the \emph{profile} of a matrix. The matrix profile provides the key link between corner polyhedra and the special Hessenberg functions that appear in Theorem \ref{thm_main}. We show in Corollary \ref{cor_bijection} that the profile of a matrix exactly describes its corner polyhedron.

\begin{defn} \label{def_profile}
A \emph{path} in an $(N+1) \times (N+1)$ matrix is a word of length $2N + 2$ in the letters $\pathd$ and $\pathr$, consisting of $N + 1$ instances of each letter. We identify each path with the corresponding sequence of down-directed and right-directed line segments in the gridlines between entries. A \emph{Dyck path} is a path that does not cross the main diagonal, but may meet it. A \emph{lower Dyck path} is a Dyck path beginning with $\pathd$, and an \emph{upper Dyck path} is a Dyck path beginning with $\pathr$. Note that the first letter determines whether the Dyck path is below or above the main diagonal.

A path \emph{contains} a pair $(i,j) \in [N+1]^2$ if $(i,j)$ is above the path.

The \emph{profile} of a matrix $M$, denoted $\Prof(M)$, is defined as follows.

\begin{enumerate}
    \item Suppose that $M$ is not strictly-upper-triangular. Then $\Prof(M)$ is defined to be the upper Dyck path in the grid lines of $M$ that contains the fewest entries among all upper Dyck paths that contain all nonvanishing entries.
    \item Suppose that $M$ is strictly upper-triangular. Then $\Prof(M)$ is defined as the lower Dyck path that contains the fewest entries among all lower Dyck paths that contain all nonvanishing entries.
\end{enumerate}
We call $\Prof(0)$ the \emph{trivial profile}; we treat it as an exceptional case.
\end{defn}

\begin{remark}
For strictly upper-triangular matrices $M$, the profile $\Prof(M)$ is exactly the data of the Hessenberg function (Definition \ref{defn:flaginkN1}) associated to $M$ with respect to the standard complete flag. For all other $M$, the profile of $M$ is the data of the pointwise maximum of the identity function and that Hessenberg function.
\end{remark}

\begin{defn} \label{def_pivotal_entry}
A \emph{pivotal entry} of a matrix $M$ is a pair $(i,j)$, where $i, j \in [N+1]$, such that the path $\Prof(M)$ follows the closest gridlines to the left of $(i,j)$ and below $(i,j)$. In other words, the pivotal entries correspond to appearances of the subword $\pathd \pathr$ in $\Prof(M)$.
\end{defn}

\begin{lemma} \label{lem_minimal_entries}
Let $M \in \Mat_{N+1}$ be a nonzero matrix. Let $E(M)$ denote the set of nonvanishing entries $(i, j) \in [N+1]^2$ of $M$ that are minimal with respect to outweighing (Definition \ref{def_outweighing}) among all nonvanishing entries of $M$. Then $E(M)$ is a subset of the set of pivotal entries of $M$, as follows:
\begin{enumerate}
    \item If $M$ is upper-triangular but not strictly-upper-triangular, then $E(M)$ is the set of nonvanishing diagonal entries of $M$.
    \item If $M$ is not upper-triangular, then $E(M)$ is the set of pivotal entries $(i,j)$ of $M$ such that $i > j$.
    \item If $M$ is strictly upper-triangular, then $E(M)$ is the set of pivotal entries of $M$.
\end{enumerate}
\end{lemma}

\begin{proof}
\par(1)\enspace Since $M$ is upper-triangular but not strictly-upper-triangular, it has a nonvanishing diagonal entry, so by Lemma \ref{lem_outweighing}, no entry $(i, j)$ with $i < j$ is minimal. Since $M$ is upper-triangular, by Lemma \ref{lem_outweighing}, all the diagonal nonvanishing entries are minimal.

\par(2) \enspace Since $M$ is not upper-triangular, it has a nonvanishing entry below the main diagonal. By Lemma \ref{lem_outweighing}, that entry outweighs every entry on or above the main diagonal. So every element of $E(M)$ is below the main diagonal. Consider a nonvanishing entry $(i',j')$ below the main diagonal. The entry $(i', j')$ is outweighed by some entry $(i, j)$ if and only if $(i',j')$ is not pivotal.

\par(3) \enspace Consider an entry $(i',j')$ above the main diagonal. Since $M$ is strictly upper-triangular, by Lemma \ref{lem_outweighing}, the entry $(i',j')$ is outweighed by some entry $(i,j)$ if and only if $(i',j')$ is not pivotal.
\end{proof}

The set of pivotal entries of $M$ has a natural ordering, namely the order in which the pivotal entries appear along the path from upper-left to bottom-right.

Recall that the set $\Ctrl(M)$ of \emph{control points} is the set of vertices of the corner polyhedron $\Pi_+(M)$. The next proposition is the key result of this section. It shows that control points essentially correspond to pivotal entries of $M$.

\begin{prop} \label{prop_vertices}
For any nonzero matrix $M \in \Mat_{N+1}$, there is a canonical ordering on the elements of $\Ctrl(M)$.
\begin{itemize}
    \item If $M$ is upper triangular but not strictly-upper-triangular, then $\Ctrl(M) = \{0\}$, so the ordering is the trivial one.
    \item For any other nonzero $M$, the off-diagonal pivotal entries of $M$ are in bijection with $\Ctrl(M)$ by the map
    $$(i,j) \mapsto \barwt(a_{ij}).$$
    The canonical ordering on $\Ctrl(M)$ can be defined in any of the following equivalent ways:
    \begin{itemize}
        \item Take the ordering inherited from the natural ordering on the pivotal entries;
        \item Take the ordering by the index of the first nonzero coordinate;
        \item Take the ordering by the index of the last nonzero coordinate.
    \end{itemize}
\end{itemize}
\end{prop}

\begin{proof}
Let $E(M)$ be as defined in Lemma \ref{lem_minimal_entries}. Let
$$\barwt(E(M)) = \{ \barwt(a_{ij})  \colon (i,j) \in E(M) \}.$$

We consider the same three cases. 

\par(1)\enspace Suppose $M$ is upper triangular but not strictly-upper-triangular. Then it follows immediately from Lemma \ref{lem_root_weights} that $\Ctrl(M) = \{0\}$.

\par(2)\enspace Suppose $M$ is not of the form of case (1). We claim that $\Ctrl(M) = \barwt(E(M))$. The containment $\Ctrl(M) \subseteq \barwt(E(M))$ follows from the definition of $E(M)$, since if some entry $(i,j)$ is outweighed by a nonvanishing entry $(i', j')$, then $\barwt(a_{ij}) \in \barwt(a_{i' j'}) + \Orth$ and $\barwt(a_{ij}) \neq \barwt(a_{i'j'})$, so $\barwt(a_{ij})$ is not a vertex.

By Lemma \ref{lem_minimal_entries}, there is an induced ordering on $E(M)$ from the canonical ordering on the pivotal entries. Since $\barwt$ is injective on $E(M)$, we also obtain an ordering on $\barwt(E(M))$, and by restriction, on $\Ctrl(M)$. This is the canonical ordering in the theorem statement. The canonical ordering on the set of pivotal entries can equivalently be described as ordering by increasing row or increasing column in $M$. Combining this observation with the explicit description of the weights in Lemma \ref{lem_root_weights} gives us the three equivalent definitions of the ordering on $\Ctrl(M)$.

We claim $\Ctrl(M) \supseteq \barwt(E(M))$. To prove this, we subdivide into two cases.
\par(2a)\enspace Assume $M$ is non-upper-triangular. Then $\Prof(M)$ is a lower Dyck path.

Given $w \in \barwt(E(M))$, we must show that $w$ is a vertex of $\Pi_+(M)$. This is true if there are $N$ independent linear functionals $r \in \bar{R}^+$ which achieve their minimum on $\Pi_+(M)$ at $w$. Say $\kappa = \# \barwt(E(M))$. Using the canonical ordering, let $w_c$ be the $c$-th element of $\barwt(E(M))$. It is of the form
$$(0,\hdots,0,-1,\hdots,-1,0,\hdots,0),$$
where, for some $t_c, u_c \in [N]$ with $t_c \leq u_c$, the $-1$'s run from index $t_c$ to $u_c$ inclusive. It follows from the canonical ordering that, for any $c', c'' \in [\kappa]$ with $c' < c''$, we have $t_{c'} < t_{c''}$ and $u_{c'} < u_{c''}$. Therefore, the following $N$ functionals on $\barwt(E(M))$ are minimized at $w_c$.
\begin{itemize}
    \item For all $i$ in the range $1 \leq i < t_c$, the functional
    $s_i + s_{u_{c}}$ on $\barwt(E(M))$ achieves its minimum value $-1$ at $w_c$.
    \item For all $i$ in the range $t_c \leq i \leq u_c$, the functional
    $s_i + s_{u_c}$
    on $\barwt(E(M))$ achieves its minimum value $-2$ at $w_c$.
    \item For all $i$ in the range $u_c < i \leq N$, the functional
    $s_{t_c} + s_i$
    on $\barwt(E(M))$ achieves its minimum value $-1$ at $w_c$.
\end{itemize}
The $N$ functionals named above are easily seen to be independent, so $w_c$ is a vertex of the corner hull $\Pi_+(M)$ of $\barwt(E(M))$.

\par(2b)\enspace Assume $M$ is strictly upper-triangular. Then $\Prof(M)$ is an upper Dyck path. As in (2a), we claim $\Ctrl(M) \supseteq \barwt(E(M))$. Given $w \in \barwt(E(M))$, we seek $N$ independent linear functionals $r \in \bar{R}^+$ which achieve their minimum on $\Pi_+(M)$ at $w$. Say $\kappa = \# \barwt(E(M))$. Using the canonical ordering, let $w_c$ be the $c$-th element of $\barwt(E(M))$. It is of the form
$$(0,\hdots,0,1,\hdots,1,0,\hdots,0),$$
where for some $t_c, u_c \in [N]$ with $t_c \leq u_c$, the $1$'s run from index $t_c$ to $u_c$ inclusive. It follows from the canonical ordering that, for any $c', c'' \in [\kappa]$ with $c' < c''$, we have $t_{c'} < t_{c''}$ and $u_{c'} < u_{c''}$. Then the following $N$ functionals are independent and simultaneously minimized at $w_c$.
\begin{itemize}
    \item For all $1 \leq i < t_c$ and $i > u_c$, the functional $s_i$ on $\barwt(E(M))$ achieves its minimum value $0$ at $w_c$.
    \item For all $i$ in the range $t_c \leq u_c$, the functional
    $$s_1 + \hdots + s_{t_c - 1} + s_i + s_{u_c + 1} + \hdots + s_\kappa$$
    achieves its minimum value 1 at $w_c$.
\end{itemize}
\end{proof}

Finally, we classify corner polyhedra of type $A_N$.
\begin{cor} \label{cor_bijection}
    Two nonzero matrices in $\Mat_{N+1}$ have the same corner polyhedron if and only if they have the same profile.
\end{cor}

\begin{proof}
 Let $M_1, M_2 \in \Mat_{N+1}$. By the basic properties of the corner hull operation, we have $\Pi_+(M_1) = \Pi_+(M_2)$ if and only if $\Ctrl(M_1) = \Ctrl(M_2)$. By Lemma \ref{prop_vertices}, two matrices $M_1$ and $M_2$ satisfy $\Ctrl(M_1) = \Ctrl(M_2)$ if and only if $\barwt(E(M_1)) = \barwt(E(M_2))$. The set $\barwt(E(M))$ is determined by the off-diagonal pivotal entries of $M$, which is determined by $\Prof(M)$.
\end{proof}

\begin{cor} \label{cor_catalan}
    Let $C_N$ be the $N$-th Catalan number. The number of corner polyhedra of type $A_N$ is $2 C_{N+1} - 1$. 
\end{cor}

\begin{proof}
 The corner polyhedra of type $A_N$ all arise as corner polyhedra associated to matrices by Corollary \ref{cor_bijection}, and so it suffices to count the set 
 $$\{ \Prof(M)  \colon 0 \neq M \in \Mat_{N+1}\}$$
 instead. Any lower Dyck path and nontrivial upper Dyck path may appear as a matrix profile. By definition, the Catalan numbers count Dyck paths, so there are $C_{N+1}$ lower Dyck paths and $C_{N+1} - 1$ nontrivial upper Dyck paths.
\end{proof}

For another connection between Catalan numbers and $A_N$ root polytopes, see \cite[Theorem 8]{MR1697418}.

\subsection{Facets of corner polyhedra of type \texorpdfstring{$A_N$}{A N}} \label{sec_facets}
We have calculated the vertices of $\Pi_+(M)$ for any nonzero matrix $M \in \Mat_{N+1}$ (Proposition \ref{prop_vertices}). In this section, we calculate the facets of $\Pi_+(M)$ (Proposition \ref{prop_facets}). These results are the key input for proving Theorem \ref{thm_main}; the facets of $\Pi_+(M)$ roughly correspond to Type I, II, and III flags relative to the linear map $M$.

We work with halfspaces in $\bar{W}$, which are of the form $\{ w \in \bar{W}  \colon s(w) \geq c \}$ for some $s \in \bar{R}^+, c \in \R$. For convenience, we write halfspaces as $\{ s \geq c \}$. 

For any nonempty $I \subseteq [N]$, we define a functional
$$s_I := \sum_{i \in I} s_i.$$

\begin{defn}
The \emph{standard flag in $k^{N+1}$} is the sequence 
$$0 = S_0 \subsetneq S_1 \subsetneq \hdots \subsetneq S_{N+1} = k^{N+1}$$ of subspaces of $k^{N+1}$, where for each $\rho \in [N]$, the subspace $S_\rho$ is the span of the first $\rho$ standard basis vectors. 

The \emph{Hessenberg function} of $M$, denoted $\Hess_M$, is the Hessenberg function of the linear map defined by $M$ on $k^{N+1}$ relative to the standard flag:
\begin{align*}
    \Hess_{M}  : \; & \{0, \hdots, N + 1\} \to \{0, \hdots N + 1\}, \\
    & i \mapsto \min \{j : MS_i \subseteq S_j\}.
\end{align*}
Each nonempty subset 
$$I = \{i_1, \hdots, i_\gamma\} \subseteq [N]$$
determines a subflag 
$$0 = S_{i_0} \subsetneq S_{i_1} \subsetneq \hdots \subsetneq S_{i_\gamma} \subsetneq S_{i_{\gamma + 1}} = k^{N+1}$$
of the standard flag. Let $h_{M, I}$ denote the Hessenberg function of $M$ with respect to that subflag. Explicitly,
\begin{align*}
    \Hess_{M, I}  : \; & \{0, \hdots, \gamma + 1 \} \to \{ 0, \hdots, \gamma + 1 \}, \\
    & t \mapsto \min \{t' : MS_{i_t} \subseteq S_{i_{t'}}\}.
\end{align*}
\end{defn}

The following proposition is the main result of this section.

\begin{prop} \label{prop_facets}
Let $M \in \Mat_{N + 1}$ be a nonzero matrix. Let $0 = S_0 \subset S_1 \subset \hdots \subset S_{N+1}$ be the standard flag. The facets of the corner polyhedron $\Pi_+(M)$ are of four kinds, called (F-1A), (F-1B), (F-2A), and (F-2B), as follows. For each nonempty subset $$I = \{i_1, \hdots, i_\gamma\} \subseteq [N],$$
where by convention we let $\gamma = \#I$ and
$$0 = i_0 < i_1 < \hdots < i_\gamma < i_{\gamma + 1} = N + 1,$$
there is at most one facet of $\Pi_+(M)$.

If $M$ is not strictly-upper-triangular:
        \begin{itemize}
        \item (F-1A): If 
            \begin{equation*}
            \label{eq_1A}
                \gamma = 1 \quad \text{and} \quad M S_{i_1} \subseteq S_{i_1},
            \tag{\ref{prop_facets}.1A}
            \end{equation*}
        there is a facet 
            $$\{ s_I \geq 0 \}.$$
        \item (F-1B): If
            \begin{equation*} \label{eq_1B}
            h_{M, I}(t) = t + 1 \quad (1 \leq t \leq \gamma),
            \tag{\ref{prop_facets}.1B}
            \end{equation*}
        there is a facet
            $$ \left\{ s_I \geq -1 \right\}. $$
        \end{itemize}
If $M$ is strictly upper-triangular:
    \begin{itemize}
        \item (F-2A): If
            \begin{equation*} \label{eq_2A}
            \gamma = 1 \quad \text{ and } \quad (MS_{N+1} \not\subseteq S_{i_1} \quad \text{or} \quad MS_{i_1} \neq 0),
            \tag{\ref{prop_facets}.2A}
            \end{equation*}
        there is a facet
        $$\{ s_I \geq 0 \}.$$

        \item (F-2B): If
            \begin{equation*} \label{eq_2B}
            h_{M, I}(t) = t - 1 \quad (1 \leq t \leq \gamma + 1),
            \tag{\ref{prop_facets}.2B}
            \end{equation*}
        there is a facet
        $$\{ s_I \geq 1\}.$$
    \end{itemize}
\end{prop}
The proof of Proposition \ref{prop_facets} is involved. We first prove intermediate results, Lemmas \ref{lem_facets_lower_ctrl} and \ref{lem_facets_upper_ctrl}, that compute facets of $\Pi_+(M)$ via the following combinatorial gadget.

\begin{defn}
For any nonzero matrix $M \in \Mat_{N+1}$, we view $\Ctrl(M)$ as a matrix, the \emph{control matrix} of $M$, by ordering its elements according to the canonical ordering of Proposition \ref{prop_vertices}. Let $\kappa = \# \Ctrl(M)$. Then $\Ctrl(M)$ is a $N \times \kappa$ matrix. 
For examples, see Section \ref{sect_examples}.
\end{defn}

The form of $\Ctrl(M)$ is one of three kinds.
\begin{enumerate}
    \item If $M$ is upper-triangular, but not strictly so, then $\Ctrl(M) = 0.$
    \item If $M$ is not upper-triangular, then the entries of $\Ctrl(M)$ are each 0 or $-1$. For each $c \in [\kappa]$, say the $c$-th off-diagonal pivotal entry of $M$ is $(t_c, u_c - 1)$. Then column $c$ of $\Ctrl(M)$ consists of a single block of $-1$'s running from $t_c$ to $u_c$ inclusive, and has all other entries 0.  By Proposition \ref{prop_vertices}, if $c', c'' \in [\kappa]$ satisfy $c' < c''$, then $t_{c'} < t_{c''}$ and $u_{c'} < u_{c''}$.
    \item If $M$ is strictly upper-triangular, then the entries of $\Ctrl(M)$ are each $0$ or $1$. For each $c \in [\kappa]$, say the $c$-th off-diagonal pivotal entry of $M$ is $(u_c, t_c - 1)$; note the difference from the previous case. Then column $c$ of $\Ctrl(M)$ consists of a single block of $1$'s running from $t_c$ to $u_c$ inclusive, and has all other entries 0. By Proposition \ref{prop_vertices}, if $c', c'' \in [\kappa]$ satisfy $c' < c''$, then $t_{c'} < t_{c''}$ and $u_{c'} < u_{c''}$.
\end{enumerate}

In order to understand facets from the control matrix, we look at certain submatrices, which themselves correspond to lower-dimensional corner polyhedra. We set this argument up as follows.

Let $y_1, \hdots, y_N \in \bar{W}$ be the dual vectors to $s_1, \hdots, s_N$. Given any subset $I \subseteq [N]$, let $\bar{W}_I$ be the subspace of $\bar{W}$ generated by the $y_i, i \in I$. If $i' \not\in I$, then $s_{i'}$ is identically 0 on $\bar{W}_I$. Let $\gamma = \#I$. Taking $s_i$ for each $i \in I$ as coordinates identifies $\bar{W}_I$ with $\R^{\gamma}$. Define $\Pi_+^I(M)$ to be the projection of $\Pi_+(M)$ onto $\bar{W}_I$.

Observe that $\Pi_+^I(M)$ is itself a corner polyhedron in $\bar{W}_I$ with this choice of basis. We denote the set of control points of $\Pi_+^I(M)$ in $\bar{W}_I$ by $\VertPlusI (M)$. It follows that there is a canonical ordering on $\VertPlusI (M)$, so writing $\kappa' = \# \VertPlusI(M)$, the set $\VertPlusI(M)$ can be identified with a $\gamma \times \kappa'$ matrix.  

The matrix $\VertPlusI(M)$ can be obtained from $\Ctrl(M)$, as follows. Take the $\gamma \times \kappa$ minor of $\Ctrl(M)$ consisting of the rows with indices in $I$. If there are two columns $w_1, w_2$ in the minor with distinct indices such that $w_1 \in w_2 + O_+$ in $\bar{W}_I$, then remove $w_1$. Repeat the previous step until no such pairs exist, and the result is $\VertPlusI(M)$.

In Lemma \ref{lem_facets_lower_ctrl}, we show how the facets of $\Pi_+(M)$ may be read off the control matrix for certain matrices, an intermediate step in proving Proposition \ref{prop_facets}.

\begin{lemma} \label{lem_facets_lower_ctrl}
Let $M \in \Mat_{N+1}$ be not strictly-upper-triangular. The facets of the corner polyhedron $\Pi_+(M)$ are as follows.
\begin{itemize}
\item (F-1A): For each all-zero row $\rho$ of $\Ctrl(M)$, there is a facet
$$\{ s_\rho \geq 0\}.$$

\item (F-1B): For each nonempty subset $I \subseteq [N]$ such that $\VertPlusI(M) = -1$, there is a facet
$$\{ s_I \geq -1\}.$$
\end{itemize}
\end{lemma}

\begin{proof}
First, we check that these halfspaces are facets of $\Pi_+(M)$. 

If we have an all-zero row $\rho$ of $\Ctrl(M)$, then writing $I = \{\rho\}$, we find $\VertPlusI(M) = 0 \in \Mat_1$. Thus $\Pi_+^I(M)$ has one facet $F$, defined by $\{s_\rho \geq 0\}$. We can see that $\{ s_\rho \geq 0\}$ is a supporting halfspace of $\Pi_+(M)$. To check that it defines a facet, we show
$$\dim (\Pi_+(M) \cap \{ s_\rho = 0 \}) \geq N - 1.$$
To see this, notice that the functional $s_\rho$ is identically $0$ on the convex cone $F'$ generated by $y_i, i \not\in I$. Since $F$ and $F'$ are orthogonal, we have $\dim F + F' = N - 1$, and 
$$F + F' \subseteq \Pi_+(M) \cap \{ s_\rho = 0\}.$$

Suppose that a nonempty subset $I \subseteq [N]$ satisfies $\VertPlusI(M) = -1 \in \Mat_\gamma$. Then $\{s_I = -1\}$ contains $\gamma$ vertices of $\VertPlusI(M)$, and $\Pi^I_+(M) \subset \{s_I \geq -1\}$, so $\{s_I \geq -1\}$ defines a facet $F$ of $\Pi_+^I(M)$. Thus $\{s_I \geq -1\}$ is a supporting halfspace of $\Pi_+(M)$. To check that it defines a facet, we show
$$\dim(\Pi_+(M) \cap \{s_I = -1\}) \geq N - 1.$$
To see this, notice that the functional $s_I$ is identically 0 on the convex cone $F'$ generated by $y_i, i \not\in I$. Since $F$ and $F'$ are orthogonal, we have $$\dim F + F' = \gamma - 1 + (N - \gamma) = N - 1,$$
and $$F + F' \subseteq \Pi_+(M) \cap \{s_I = -1\}.$$

It remains to show that every facet of $\Pi_+(M)$ is of one of these two forms. Suppose that $\{ r > c \}$ is a facet, where $r \in \bar{R}^+$ is a nonzero functional and $c \in \R$. Let $I \subseteq [N]$ be the set of indices $i \in [N]$ such that $r(y_i) \neq 0$. Since $r \neq 0$, the set $I$ is nonempty. Projecting to $\bar{W}_I$, the halfspace $\{ r > c \}$ defines a facet $F$ of $\Pi_+^I(M)$. We claim that $\Pi_+^I(M)$ has $\#I$ vertices. Since $\VertPlusI(M)$ has $\#I$ rows, by the canonical ordering, it has at most $\#I$ columns, so $\Pi_+^I(M)$ has at most $\#I$ vertices. On the other hand, the facet $F$ is the Minkowski sum of some face $F''$ of $\Orth$ in $\bar{W}_I$ and a face of $\Conv(\VertPlusI(M))$. But $F''$ is necessarily the origin, by construction of the set $I$. It follows that $F$ is a facet of the convex polytope $\Conv(\VertPlusI(M))$. Since $\dim F = \#I - 1$, at least $\#I$ elements of $\VertPlusI(M)$ lie on $F$, supplying the claimed number of vertices of $\Pi_+^I(M)$.

Since $\VertPlusI(M)$ is a square matrix, it follows from the canonical ordering that either $\VertPlusI(M) = 0$ or $\VertPlusI(M) = -1 \in \Mat_\gamma$, where $\gamma = \# I$. If $\VertPlusI(M) = 0$, then $I$ is some singleton $\{\rho\}$, and this determines the facet; we conclude $\{r \geq c\} = \{s_\rho \geq 0\}$. If $\VertPlusI(M) = -1$, then the facet containing the necessary $\gamma$ vertices is given explicitly by $\{r \geq c\} = \{s_i \geq -1\}$.
\end{proof}

In Lemma \ref{lem_facets_upper_ctrl}, we show how the facets of $\Pi_+(M)$ may be read off the control matrix for strictly-upper-triangular matrices, another step in proving Proposition \ref{prop_facets}.
\begin{lemma} \label{lem_facets_upper_ctrl}
Let $M \in \Mat_{N+1}$ be strictly upper-triangular and nonzero. The facets of the corner polyhedron $\Pi_+(M)$ are as follows. 
\begin{itemize}
\item (F-2A): For each $\rho \in [N]$ such that $\VertPlusI(M) = 0 \in \Mat_1$, there is a facet
$$\{s_\rho \geq 0\}.$$
\item (F-2B): For each nonempty subset $I \subseteq [N]$ such that $\VertPlusI(M) = 1 \in \Mat_{\#I}$, there is a facet
$$\{s_I \geq 1 \}.$$
\end{itemize}
\end{lemma}

\begin{proof}
The argument resembles the proof of Lemma \ref{lem_facets_lower_ctrl}. First, we check that these halfspaces are facets of $\Pi_+(M)$.

Suppose that $I = \{\rho\}$ satisfies $\VertPlusI(M) = 0 \in \Mat_1$. By the proof of Lemma \ref{lem_facets_lower_ctrl}, the supporting halfspace $\{s_\rho \geq 0 \}$ defines a facet of $\Pi_+(M)$.

Suppose that a nonempty subset $I \subseteq [N]$ satisfies $\VertPlusI(M) = 1 \in \Mat_\gamma$, where $\gamma = \#I$. Then $\{s_I = 1\}$ contains all $\gamma$ vertices of $\Pi_+^I(M)$, so again by the proof of Lemma \ref{lem_facets_lower_ctrl}, there is a facet of $\Pi_+(M)$ defined by $\{s_I \geq 1\}$.

It remains to show that every facet of $\Pi_+(M)$ is of one of these two forms. By the proof of Lemma \ref{lem_facets_lower_ctrl}, if $\{r > c\}$ defines a facet of $\Pi_+(M)$, and $I = \{i \in [N] \colon r(y_i) \neq 0\}$, then $\VertPlusI(M)$ is a square matrix. It follows from the canonical ordering that either $\VertPlusI(M) = 0$ or $\VertPlusI(M) = 1 \in \Mat_\gamma$, where $\gamma = \#I$. If $\VertPlusI(M) = 0$, then $I$ is some singleton $\{\rho\}$, and this determines the facet; we conclude $\{r \geq c\} = \{s_\rho \geq 0\}$. If $\VertPlusI(M) = 1$, then the facet containing the necessary $\gamma$ vertices is given explicitly by $\{r \geq c\} = \{s_I \geq 1\}$.
\end{proof}

Finally, we prove the main result of this section.

\begin{proof}[Proof of Proposition \ref{prop_facets}]
By Lemma \ref{lem_facets_lower_ctrl} and Lemma \ref{lem_facets_upper_ctrl}, it suffices to prove the equivalence of the conditions for each facet type.

Recall the definition of $E(M)$ from Lemma \ref{lem_minimal_entries}. Observe that, for any not strictly-upper-triangular $M$, and for any $\rho \in [N]$ and $c \in [\kappa]$,
\[
\Ctrl(M)_{\rho c} = 
\begin{cases}
-1, & \text{the } c \text{-th element of } E(M) \text{ is in rectangle } [\rho + 1 , N + 1] \times [1,  \rho],\\
0, & \text{otherwise}.
\end{cases}
\]
Similarly, for any strictly-upper-triangular $M$,
\[
\Ctrl(M)_{\rho c} =
\begin{cases}
1, & \text{the } c \text{-th element of } E(M) \text{ is in rectangle } [1, \rho] \times [\rho + 1, N + 1],\\
0, & \text{otherwise}.
\end{cases}
\]
    \par(F-1A)\enspace We must show that, for a given nonempty subset $I \subseteq [N]$, the condition \eqref{eq_1A} is equivalent to $\VertPlusI(M) = 0 \in \Mat_1$. Since $\gamma = 1$ is equivalent to $\VertPlusI$ having exactly one row, we must show that under this condition, 
    $$\VertPlusI(M) = [0] \Leftrightarrow MS_{i_1} \subseteq S_{i_1} .$$
    
    Suppose that $\VertPlusI(M) = [0]$. By the canonical ordering, row $i_1$ of $\Ctrl_M$ is all zeroes. By the observation, there are no pivotal entries of $M$ in rectangle $[i_1 + 1, N + 1] \times [1, i_1]$. Then, all entries of $M$ in that rectangle must be 0. It follows that $MS_{i_1} \subseteq S_{i_1}$.
    
    For the reverse implication, suppose that $MS_{i_1} \subseteq S_{i_1}$. All entries of $M$ in the rectangle $[i_1 + 1, N + 1] \times [1, i_1]$ are zero, so in particular, there are no pivotal entries of $M$ in that rectangle. Then, by the observation, we deduce that $\Ctrl(M)$ has an all-0 row $i_1$, so $\VertPlusI(M) = [0]$.
    
    \par(F-1B)\enspace For each nonempty $I \subseteq [N]$, we must prove the equivalence 
    $$\VertPlusI(M) = -1 \Leftrightarrow \eqref{eq_1B}.$$
    
    If $I$ satisfies \eqref{eq_1B}, we claim there at least $\#I$ elements of $\VertPlusI(M)$. To each $i \in I$, we can associate an element of $E(M)$, as follows. In column $i$ of $M$, let $(i', i)$ be the bottom-most nonzero entry; then consider the left-most nonzero entry in row $i'$; call this entry $\iota(i)$. The entry $(i', j')$ is in $E(M)$ since the path $\Prof(M)$ borders it to the left and below. By \eqref{eq_1B}, the map $\iota \colon I \to E(M)$ is injective. We may compute the entries of $\VertPlusI(M)$ using the observation to directly calculate that $\VertPlusI(M) = -1$.
    
    Suppose that $I$ does not satisfy \eqref{eq_1B}. Then there is some $i \in I$ such that $S_i = MS_i$ or $MS_i \not\subseteq S_{i'}$, where $i'$ is the next element in $I$. In the first case, the rectangle $[i+1, N + 1] \times [1, i]$ of $M$ is all zero. By the observation, row $i$ of $\VertPlusI(M)$ is all 0, so $\VertPlusI(M) \neq -1$. In the second case, there exists a nonzero entry of $M$ in $[i' + 1, N + 1] \times [ 1, i]$, so there is a pivotal entry in that rectangle. By the observation, the corresponding column of $\Ctrl(M)$ has a $-1$ in row $i$ and $i'$, so $\VertPlusI(M) \neq -1$.
    
    \par(F-2A)\enspace We must show that, for nonzero strictly-upper-triangular $M$,
    $$\VertPlusI(M) = [0] \Leftrightarrow \eqref{eq_2A}.$$
    Suppose that $\VertPlusI(M) \neq [0]$. By the canonical ordering, row $i_1$ of $\Ctrl(M)$ is all 1's. By the observation, every pivotal entry of $M$ is in the rectangle $[1, i_1] \times [i_1 + 1, N + 1]$. Then every nonzero entry of $M$ is in that rectangle, so \eqref{eq_2A} does not hold. For the reverse implication, if \eqref{eq_2A} does not hold, then every entry of $M$ is in that rectangle, so every pivotal entry of $M$ is in that rectangle, so $\VertPlusI(M) = [1].$
    
    \par(F-2B)\enspace We must show that, for nonzero strictly-upper-triangular $M$, for each nonempty $I \subseteq [N]$,
    $$\VertPlusI(M) = 1 \Leftrightarrow \eqref{eq_2B}.$$
    Suppose that $\VertPlusI(M) = 1$. Then, using the observation, in the notation of Proposition \ref{prop_facets}, for each $0 \leq t \leq \#I$, the rectangle $[i_t + 1, N + 1] \times [1, i_{t+1}]$ of $M$ has no pivotal entries, and for each $1 \leq t \leq \#I$, there is a pivotal entry of $M$ in $[i_{t - 1} + 1, i_t] \times [1, i_{t+1}]$. It follows that the rectangle $[i_t + 1, N + 1] \times [1, i_{t+1}]$ is all 0's and there is a nonzero entry in $[i_{t - 1} + 1, i_t] \times [1, i_{t+1}]$, implying \eqref{eq_2B}.
    
    Conversely, suppose that \eqref{eq_2B} holds. For each $0 \leq t \leq \#I$, the rectangle $[i_t + 1, N + 1] \times [1, i_{t+1}]$ is all 0's, and for each $1 \leq t \leq \#I$, there is a nonzero entry in $[i_{t - 1} + 1, i_t] \times [1, i_{t+1}]$. It follows that the map $\iota$ of (F-1B) is injective, so there are at least $\#I$ elements of $\Ctrl (M)$. Using the observation, we calculate explicitly that $\VertPlusI(M) = 1.$
\end{proof}

\subsection{Examples} \label{sect_examples}

In the following examples, we fix some $N \geq 1$ and consider some $(N + 1) \times (N + 1)$ matrix $M$. We mark nonzero entries by $\bm{\ast}$, and entries with no conditions on them by $\ast$.

\begin{example} \label{ex_generic}
Suppose that entry $(N+1, 1)$ of $M$ is nonzero. Then the profile $\Prof(M)$ is the lower Dyck path
\[
\left[
\begin{array}{ccc}
\stairbarast & \hdots & \ast \\
\stairbarvdots & \ddots & \vdots \\
\piv & \hdots & \ast \\ \cline{1-3}
\end{array}
\right].
\]
By Proposition \ref{prop_vertices}, the control matrix $\Ctrl(M)$ is the $N \times 1$ matrix with every entry equal to $-1$. Thus the facets of $\Pi_+(M)$ are, for each $i \in [N]$,
$$\{ s_i \geq -1 \}.$$ 
This is obvious because $\Ctrl(M)$ has just one column, but we can also see it as a special case of Proposition \ref{prop_facets} in which the facets are all of type (F-1B).

This is the least degenerate matrix $M$, and the corner polyhedron is as simple as can be. But note that, even though a generic matrix is described by this example, a generic linear transformation $T$ is not, since there is always some basis in which $T$ is of a different form.
\end{example}

\begin{example} \label{ex_most_degen}
Suppose that the only nonzero entry of $M$ is $(1, N + 1)$. Then the profile $\Prof(M)$ is the upper Dyck path
\[
\left[
\begin{array}{cccc}
\cline{1-3}
0 & \hdots & 0 & \piv \\ \cline{4-4}
0 & \ddots & 0 & \stairzerobar \\
\vdots & \ddots & \ddots  & \stairvdotsbar \\
0 & \hdots & 0 & \stairzerobar \\
\end{array}
\right].
\]
We can think of this matrix $M$ as the worst-case scenario for stability. By Proposition \ref{prop_vertices}, the control matrix $\Ctrl(M)$ is the $N \times 1$ matrix with every entry equal to 1. Since there is only one column of $\Ctrl(M)$, the corner polyhedron $\Pi_+(M)$ has facets $\{s_i \geq 1\}$, for each $i \in [N]$. As a special case of Proposition \ref{prop_facets}, the facets are all of type (F-2B).
\end{example}

\begin{example} \label{ex_main_stair}
Suppose that $M$ is upper-triangular and entry $(1,1)$ is nonzero. This example includes the case $M = 1$. Since there is a nonzero entry on the main diagonal, the profile $\Prof(M)$ is a lower Dyck path, specifically
\[
\left[
\begin{array} {cccc}
\piv & \ast & \hdots & \ast \\ \cline{1-1}
0 & \stairbarast & \hdots & \ast \\ \cline{2-2}
0 & 0 & \ddots & \ast \\
0 & 0 & \hdots & \stairbarast \\ \cline{4-4}
\end{array}
\right].
\]
By Proposition \ref{prop_vertices}, the control matrix $\Ctrl(M)$ is the $N \times 1$ matrix with every entry equal to 0. Thus $\Pi_+(M)$ is just the nonnegative orthant. The corner polyhedron $\Pi_+(M)$ has facets $\{s_i \geq 0\}$, for each $i \in [N]$. As a special case of Proposition \ref{prop_facets}, all the facets are of type (F-1A).

When $N = 2$, we can draw $\Pi_+(M)$ in the plane:
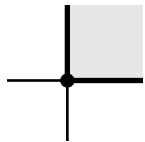
\begin{center}
\begin{tikzpicture}[line cap=round,line join=round,x=0.4cm,y=0.4cm]
\clip(-2,-2) rectangle (2.5,2.5);
\draw [line width=1pt] (-2,0)-- (3,0);
\draw [line width=1pt] (0,-2)-- (0,3);
\fill[line width=2pt,fill=gray] (0,0) -- (3,0) -- (3,3) -- (0,3) -- cycle;
\draw [line width=2pt] (0,0)-- (3,0);
\draw [line width=2pt] (3,0)-- (3,3);
\draw [line width=2pt] (3,3)-- (0,3);
\draw [line width=2pt] (0,3)-- (0,0);
\begin{scriptsize}
\draw [fill=black] (0,0) circle (2.5pt);
\end{scriptsize}
\end{tikzpicture}
\end{center}
\end{example}

Now we get to some more interesting examples.

\begin{example} \label{ex_lower_stair}
Suppose that $M$ has entry $(i + 1, i)$ nonzero for each $1 \leq i < N + 1$, and the entries $(i, j)$ with $j > i + 1$ are all 0. Then the profile of $M$ is the lower Dyck path shown: 
\[
\small
\left[
\begin{array}{ccccccc}
\stairbarast & \ast & \hdots & \ast & \ast & \ast \\
\piv & \ast & \hdots & \ast & \ast & \ast\\ \cline{1-1}
0 & \piv & \hdots & \ast & \ast & \ast \\ \cline{2-2}
\vdots & \vdots & \ddots & \vdots & \vdots & \vdots \\
0 & 0 & \hdots & \piv & \ast & \ast \\ \cline{4-4}
0 & 0 & \hdots & 0 & \piv & \ast \\ \cline{5-6}
\end{array}
\right]
\]
By Proposition \ref{prop_vertices}, we have $\Ctrl(M) = -1 \in \Mat_N$. Thus there are $N$ vertices of $\Pi_+(M)$. By Proposition \ref{prop_facets}, the facets of $\Pi_+(M)$ are, for each nonempty subset $I \subseteq [N]$, given by
$$\left\{ \sum_{i \in I} s_i \geq -1 \right\}.$$ Therefore there are $2^N - 1$ facets, all of type (F-1B).

When $N = 2$, we can draw $\Pi_+(M)$ in the plane:
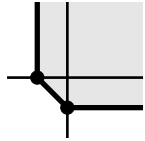
\begin{center}
\begin{tikzpicture}[line cap=round,line join=round,x=0.4cm,y=0.4cm]
\clip(-2,-2) rectangle (2.5,2.5);
\fill[line width=2pt,fill=gray] (-1,0) -- (0,-1) -- (3,-1) -- (3,3) -- (-1,3) -- cycle;
\draw [line width=1pt] (-2,0)-- (3,0);
\draw [line width=1pt] (0,-2)-- (0,3);
\draw [line width=2pt] (-1,0)-- (0,-1);
\draw [line width=2pt] (0,-1)-- (3,-1);
\draw [line width=2pt] (3,-1)-- (3,3);
\draw [line width=2pt] (3,3)-- (-1,3);
\draw [line width=2pt] (-1,3)-- (-1,0);
\begin{scriptsize}
\draw [fill=black] (-1,0) circle (2.5pt);
\draw [fill=black] (0,-1) circle (2.5pt);
\end{scriptsize}
\end{tikzpicture}
\end{center}
\end{example}

\begin{example} \label{ex_upper_stair}
Suppose that $M$ has entry $(i, i + 1)$ nonzero for each $1 \leq i < N + 1$, and all entries on or below the main diagonal are equal to $0$. Then $\Prof(M)$ is the upper Dyck path shown:
\[
\small
\left[
\begin{array}{cccccc}
\cline{1-1}
0 & \piv & \ast & \hdots & \ast & \ast \\ \cline{2-2}
0 & 0 & \piv & \hdots & \ast & \ast \\ \cline{3-3}
\vdots & \vdots & \ddots & \ddots & \vdots & \vdots \\
0 & 0 & 0 & \ddots & \piv & \ast \\ \cline{5-5}
0 & 0 & 0 & \hdots & 0 & \piv \\ \cline{6-6}
0 & 0 & 0 & \hdots & 0 & \stairzerobar
\end{array}
\right]
\]
By Proposition \ref{prop_vertices}, we have $\Ctrl(M) = 1 \in \Mat_N$. Thus there are $N$ vertices of $\Pi_+(M)$. By Proposition \ref{prop_facets}, there are $N + 1$ facets: a Type (F-2A) facet $\{s_i \geq 0\}$ for each $i \in [N]$, and a Type (F-2B) facet $\{\sum_{i \in I} s_i \geq 1\}$ corresponding to $I = [N]$. Comparing with Example \ref{ex_lower_stair}, we see that the seemingly dual cases $\Ctrl(M) = 1$ and $\Ctrl(M) = -1$ produce quite different corner polyhedra.

When $N = 2$, we can draw $\Pi_+(M)$ in the plane:
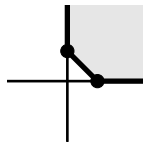
\begin{center}
\begin{tikzpicture}[line cap=round,line join=round,x=0.4cm,y=0.4cm]
\clip(-2,-2) rectangle (2.5,2.5);
\draw [line width=1pt] (-2,0)-- (3,0);
\draw [line width=1pt] (0,-2)-- (0,3);
\fill[line width=2pt,fill=gray] (0,1) -- (1,0) -- (3,0) -- (3,3) -- (0,3) -- cycle;
\draw [line width=2pt] (0,1)-- (1,0);
\draw [line width=2pt] (1,0)-- (3,0);
\draw [line width=2pt] (3,0)-- (3,3);
\draw [line width=2pt] (3,3)-- (0,3);
\draw [line width=2pt] (0,3)-- (0,1);
\begin{scriptsize}
\draw [fill=black] (0,1) circle (2.5pt);
\draw [fill=black] (1,0) circle (2.5pt);
\end{scriptsize}
\end{tikzpicture}
\end{center}
\end{example}

\begin{example}
Let $N = 7$. We present a more elaborate example of matrix with a lower Dyck path as its profile:
\[
M = 
\tiny
\left[
\begin{array}{cccccccc}
\stairbarast & \ast & \ast & \ast & \ast & \ast & \ast & \ast \\
\stairbarast & \ast & \ast & \ast & \ast & \ast & \ast & \ast \\
\piv & \ast & \ast & \ast & \ast & \ast & \ast & \ast \\ \cline{1-1}
0    & \stairbarast & \ast & \ast & \ast & \ast & \ast & \ast \\
0    & \piv & \ast & \ast & \ast & \ast & \ast & \ast \\ \cline{2-5}
0    & 0    & 0    & 0    & 0    & \stairbarzero & \ast & \ast \\ \cline{6-6}
0    & 0    & 0    & 0    & 0    & 0    & \stairbarast & \ast \\
0    & 0    & 0    & 0    & 0    & 0    & \piv & \ast \\ \cline{7-8}
\end{array}
\right].
\]
We emphasize that entry $(6,6)$ is 0, but that the profile is a lower Dyck path by definition, thus must go below entry $(6,6)$. By Proposition \ref{prop_vertices},
\[
\Ctrl(M) =
\tiny
\begin{bmatrix}
-1 & 0 & 0 \\
-1 & -1 & 0 \\
0 & -1 & 0 \\
0 & -1 & 0 \\
0 & 0 & 0 \\
0 & 0 & 0 \\
0 & 0 & -1 \\
\end{bmatrix}.
\]
By Proposition \ref{prop_facets}, there are Type (F-1A) facets
$$\{s_5 \geq 0\}, \quad \{s_6 \geq 0\},$$
and there are Type (F-1B) facets $\{ s_I \geq -1 \}$ for each 
\begin{align*}
    I \in \{
    & \{ 1 \}, \{ 2 \}, \{ 3 \}, \{ 4 \}, \{ 7 \}, \\
    & \{1, 3\}, \{1, 4\}, \{1, 7\}, \\
    & \{1, 3, 7\}, \{1, 4, 7\} \}.
\end{align*}
\end{example}

\begin{example}
Consider
\[
M = 
\tiny
\left[
\begin{array}{ccccc}
\cline{1-2}
0 & 0 & \piv & \ast & \ast \\ \cline{3-3}
0 & 0 & 0 & \piv & \ast \\ \cline{4-5}
0 & 0 & 0 & 0 & \stairzerobar \\
0 & 0 & 0 & 0 & \stairzerobar \\
0 & 0 & 0 & 0 & \stairzerobar \\
\end{array}
\right].
\]
By Proposition \ref{prop_vertices},
\[
\tiny
\Ctrl(M) =
\begin{bmatrix}
1 & 0 \\
1 & 1 \\
0 & 1 \\
0 & 0 \\
\end{bmatrix}.
\]
By Proposition \ref{prop_facets}, the Type (F-2A) facets are $\{s_i \geq 0\}$ for each $i \in \{1, 3, 4\}$, and the Type (F-2B) facets are $\{ s_I \geq 1 \}$ for each 
$I \in \{ \{2\}, \{1, 3\} \}.$

\end{example}

\section{Linear maps with marked points} \label{sect_proof}

In this section, we describe GIT stability for marked linear maps relative to any linearization. We get Theorem \ref{thm_main} and the other results of the introduction as corollaries.

To describe GIT stability for marked linear maps, we plug the facet computation of Proposition \ref{prop_facets} into the Hilbert-Mumford criterion for corner polyhedra, Proposition \ref{prop_corner_numerical}. Then we show that these stability tests can be checked on just Type I, II, and III flags, by studying the ways to embed these flags in complete flags. The underlying idea is that, any time a facet of a corner polyhedron imposes a stability test, we can find the same condition imposed by some Type I, II, or III flag. 

This technique should be compared to Mumford's stability test for point configurations, Theorem \ref{thm_classic_git}. In that proof, it was unnecessary to test stability with every complete flag, since the facet conditions were very simple and could be checked on individual subspaces. Our method here is similar, but we have to consider all four kinds of facets described in Section \ref{sec_facets}. The main difficulty is in setting up the necessary terminology to keep track of the various flag and facet types.

We recall our setting. Let $N, n \in \N$. We define $\LEnd_{N+1} \cong \A^{N^2 + 2N + 1}$ to be the space of linear endomorphisms $k^{N+1} \to k^{N+1}$. We write $\PE$ for the projectivization of the nonzero elements of $\LEnd_{N+1}$. Let $\Twip = (\PP^N)^n \times \PE$. We write its elements as
$$(T,v) = (T, v_1, \hdots, v_n).$$
Note that $\PE$ is not a group, but it nevertheless admits the conjugation action of $\SL_{N+1}$. 
We study the action \eqref{eq_acn} of $\SL_{N+1}$ on $\Twip$, which we restate for convenience:
$$A \cdot (T,v) = (ATA^{-1}, Av_1, \hdots, Av_n).$$

A \emph{flag} $\mc{H}$ in $k^{N+1}$ is a sequence, of some length $\gamma \in [N]$, of linear subspaces
$$0 \subsetneq H_1 \subsetneq \hdots \subsetneq H_\gamma \subset k^{N+1}.$$
A flag is \emph{complete} if $\gamma = N + 1$ and $H_\gamma = k^{N+1}$. Equivalently, a flag is complete if, for all $i \in [\gamma]$, we have $\dim V_i = i$.

A \emph{completion} of a flag $\mc{H}$ is a complete flag $\mc{H}'$ such that every subspace appearing in $\mc{H}$ also appears in $\mc{H}'$. A \emph{completed flag pair} $(\mc{H}, \mc{H}')$ is a pair of flags such that $\mc{H}'$ is a completion of $\mc{H}$.

Each basis $b_1, \hdots, b_{N+1}$ of $k^{N+1}$ has an \emph{associated flag}: for each $i = 1, \hdots, N + 1$, let
$$H_i = \Span(b_1, \hdots, b_i).$$

If $\mc{H}$ is a complete flag, a \emph{flag basis} is any basis which has $\mc{H}$ as its associated flag. Every complete flag has a flag basis.

Fix $T \in \PE$. We define some special types (P-1A), (P-1B), (P-2A), (P-2B) of completed flag pair, with respect to $T$. Note that some completed flag pairs do not belong to any of these types. 
\begin{itemize}
    \item A completed flag pair
    $(\mc{H}, \mc{H}')$ is \emph{Type} (P-1A) if:
    \begin{itemize}
        \item $\mc{H}$ consists of just one space $H$, and
        \item In any associated basis of $\mc{H}'$, the matrix $M$ of $T$ is not strictly-upper-triangular, and
        \item the matrix $M$ has property \eqref{eq_1A}, where $i_1 \in [N]$ is the value such that $H'_{i_1} = H$.
    \end{itemize}
    \item A completed flag pair $(\mc{H}, \mc{H}')$ is of \emph{Type} (P-1B) if:
    \begin{itemize}
        \item In any associated basis of $\mc{H}'$, the matrix $M$ of $T$ is not strictly-upper-triangular, and
        \item The matrix $M$ has property \eqref{eq_1B}, where $I$ is the set of indices determined by the inclusion of $\mc{H}$ in $\mc{H}'$.
    \end{itemize}
    \item A completed flag pair $(\mc{H},\mc{H'})$ is \emph{Type} (P-2A) if:
    \begin{itemize}
        \item $\mc{H}$ consists of just one space $H$, and
        \item The matrix $M$ of $T$ in any associated basis of $\mc{H}'$ is strictly upper triangular, and
        \item The matrix $M$ has property \eqref{eq_2A}, where $i_1 \in [N]$ is the value such that $H'_{i_1} = H$.
    \end{itemize} 
    \item A completed flag pair $(\mc{H}, \mc{H}')$ is of \emph{Type} (P-2B) if:
    \begin{itemize}
        \item In any associated basis of $\mc{H}'$, the matrix $M$ of $T$ is not strictly-upper-triangular, and
        \item The matrix $M$ has property \eqref{eq_2B}, where $I$ is the set of indices determined by the inclusion of $\mc{H}$ in $\mc{H}'$.
    \end{itemize}
\end{itemize}

We recall the three types I, II, III of a flag $\mc{H}$ relative to $T$ that were defined in Theorem \ref{thm_main}. Given a flag $\mc{H}$ and $T \in \PE$, the \emph{Hessenberg function of $\mc{H}$ relative to $T$} is defined by
\begin{align*}
    \Hess_{T,\mc{H}}  : \; & \{0, \hdots, \gamma + 1\} \to \{0, \hdots \gamma + 1\}, \\
    & i \mapsto \min \{j : TH_i \subseteq H_j\}.
\end{align*}
Given a flag $\mc{H}$ and a point $T\in\PE$, we say that~$\mc{H}$ is of \emph{Type~I,~II, or~III relative to~$T$} if it has the following properties, if any:
\begin{itemize}
    \item \textup{Type I:}  $\gamma = 1$, and $\mc{H} = (H_1)$ satisfies either
    \[
    \Bigl( 0 \neq TH_1 \subseteq H_1 \Bigr)
    \qquad\text{or}\qquad
    \Bigl(TH_1 \subseteq H_1 \quad \text{and} \quad T(k^{N+1}) \not\subseteq H_1
    \Bigr).
    \]
    \item \textup{Type II:} The Hessenberg function $\Hess_{T, \mc{H}}$ satisfies, for all $t$ in the range $1 \leq t \leq \gamma,$
$$\Hess_{T, \mc{H}}(t) = t + 1.$$
    \item \textup{Type III:} The Hessenberg function $\Hess_{T, \mc{H}}$ satisfies, for all $t$ in the range $1 \leq t \leq \gamma + 1$,
$$\Hess_{T, \mc{H}}(t) = t - 1.$$
\end{itemize}

\begin{example}
We give some examples of Type I, II, and III flags.
\begin{itemize}
    \item For any $T \in \PE$, every flag consisting of just one subspace $H$ is of Type I, II, or III. 
    \item Any non-nilpotent $T \in \PE$ has flags of Type I, taking $\mc{H} = H_1$ to be the eigenspace of any nonzero eigenvalue.
    \item If $T = 1$, then every flag is Type I, and we recover the condition \eqref{eq_mumford_template} of Mumford's example.
    \item Suppose $T \neq 1$ and $T$ is invertible. Then there is some vector $v$ such that $v \neq T(v) \neq 0$. Let $H_1 = \Span(v)$ and $H_2 = \Span(v, T(v))$. The flag 
    $H_1 \subsetneq H_2$ is Type II. 
    \item Suppose $T$ has a cyclic vector; that is, there exists $v \in k^{N+1}$ such that $v, T(v), \hdots, T^{N}(v)$ form a basis. For each $i = 1, \hdots, N$, let
    $$H_i = \Span(v, T(v), \hdots, T^{i-1}(v)).$$
    The flag $H_1 \subsetneq \hdots \subsetneq H_{N}$ is Type II. 
    \item Suppose the Jordan form of $T$ has all 0's on the diagonal and all 1's on the first off-diagonal. Then there exists $v \in k^{N+1}$, such that for all $0 \leq i < N$, we have $T^i(v) \neq 0$, and $T^N(v) = 0$. Let
    $$H_i = \Span(v, T(v), \hdots, T^{i-1}(v)).$$
    The flag $H_1 \subsetneq \hdots \subsetneq H_{N}$ is Type III.
    \item Type III flags exist only if $T$ is nilpotent. But a nilpotent map $T$ may still have flags of Type I and II.
\end{itemize}
\end{example}

\begin{remark}
One can also define Type I, II, and III flags without reference to Hessenberg functions, as follows.
\begin{itemize}
    \item \textup{Type I:}  $\gamma = 1$, and $\mc{H} = (H_1)$ satisfies either
    \[
    \Bigl( 0 \neq TH_1 \subseteq H_1 \Bigr)
    \qquad\text{or}\qquad
    \Bigl(TH_1 \subseteq H_1 \quad \text{and} \quad T(k^{N+1}) \not\subseteq H_1
    \Bigr).
    \]
    \item \textup{Type II:} For all $t$ in the range $1 \leq t \leq \gamma$, we have
    $$TH_t \not\subseteq H_t \quad \text{and} \quad TH_t \subseteq H_{t+1}.$$
    \item \textup{Type III:} With the convention that $H_{-1} = \emptyset$, for all $t$ in the range $0 \leq t \leq \gamma$, we have
    $$TH_{t+1} \subseteq H_t \quad \text{and} \quad TH_{t+1} \not\subseteq H_{t-1} .$$
\end{itemize}
\end{remark}

For any point $v \in \PP^N$ and flag $\mc{H}$, we define
$$\Omega(v, \mc{H}) := \sum_{j \in [\gamma]} \left(
    \sum_{i=1}^n \frac{-j m_i}{N+1} + \sum_{i \in [n]: \; X_{j+1}(v_i) = \hdots = X_{N+1}(v_i) = 0} m_i
     \right).$$
It arises in our calculations as follows. Recall the notation $\eta(v)$ from \eqref{eq_config_vector}. If $\mc{H'}$ is any completion of $\mc{H}$, and the spaces in $\mc{H}$ have indices corresponding to $I \subseteq [N]$ in $\mc{H}'$, we have
\begin{equation} \label{eq_omega_functional}
    \Omega(v, \mc{H}) = s_I( \eta(v) ).
\end{equation}

\begin{remark}
The function $\Omega(v, \mc{H})$ measures how much the tuple $v$ intersects the flag $\mc{H}$. The more intersections there are between points of $v$ and subspaces in $\mc{H}$, the larger the value of $\Omega(v, \mc{H})$. Notice that the weighted count $\Omega(v, \mc{H})$ is unrelated to the space $\PE$; it is solely a function of the $n$-tuple $v$ and the flag $\mc{H}$. To give an idea of how the quantity $\Omega(v, \mc{H})$ behaves, observe that the following are equivalent:
\begin{enumerate}
    \item The tuple $v \in (\PP^N)^n$ is semistable relative to $\shO(\bm{m}).$
    \item For every linear subspace $H$, we have $\Omega(v, H) \leq 0.$
    \item For every flag $\mc{H}$, we have $\Omega(v, \mc{H}) \leq 0.$
\end{enumerate}
The equivalence of (1) and (2) is just the result of Theorem \ref{thm_classic_git}. The equivalence of (2) and (3) follows from the fact that, for any flag $\mc{H}$, the sum $\Omega(v, \mc{H})$ is computed over the individual subspaces making up $\mc{H}$.
\end{remark}

Finally we prove our main result, which describes GIT stability of marked linear maps for any sheaf. We prove Theorem \ref{thm_main} as a special case.

\begin{thm} \label{thm_main_any_sheaf}
Consider $\Twip$ with the action of $\SL_{N+1}$ and sheaf $\shL = \shO(q, \bm{m})$. For each flag $\mc{H}$ in $k^{N+1}$ of Type I, II, or III, define
\[
c(T, \mc{H}) = 
\begin{cases}
0, & \mc{H} \text{ is Type I}, \\
1, & \mc{H} \text{ is Type II}, \\
-1, & \mc{H} \text{ is Type III}. \\
\end{cases}
\]
Then $(T, v) \in \Twip$ is GIT semistable relative to $\shL$ if and only if, for every Type I, II, or III flag $\mc{H}$, we have
    \begin{equation} \label{eq_flag_solo}
        \Omega(v, \mc{H}) \leq q c(T, \mc{H}).
    \end{equation}
A point $(T, v) \in \Twip$ is GIT stable relative to $\shL$ if and only if, for every Type I, II, or III flag $\mc{H}$, the inequality \eqref{eq_flag_solo} is strict.
\end{thm}

\begin{proof}
Given a basis $B$ of $k^{N+1}$, let $\Pi_+(T,v)$ denote the corner polyhedron of $(T,v)$. By Proposition \ref{prop_corner_numerical}, the point $(T,v)$ is $(B,\shL)$-semistable if and only if, in the $B$-aligned basis $B'$ of $\shL$ obtained by Segre and Veronese maps, we have 
$$0 \in \Pi_+(T,v).$$
Let $M$ be the matrix of $T$ in basis $B$, and let $\Pi_+(M)$ be the associated corner polyhedron of Type $A_N$, which is by definition the corner polyhedron of $M$ relative to $\shO(1)$. Let $\Pi_+(v)$ be the corner polyhedron of $v$ relative to $\shO(\bm{m})$. By Lemma \ref{lem_segre} and \ref{lem_veronese},
$$ \Pi_+(T,v) = \Pi_+(v) + q \Pi_+(M).$$

By Proposition \ref{prop_tuple}, the corner polyhedron $\Pi_+(v)$ has just one vertex $\eta(v)$. Thus $(B,\shL)$-semistability is equivalent to
$$0 \in q \Pi_+(M) + \eta(v).$$
Equivalently,
$$-\eta(v) \in q \Pi_+(M).$$

By Proposition \ref{prop_facets}, the polyhedron $\Pi_+(M)$ is an intersection of halfspaces of the form $\{s_I \geq c\}$, for various nonempty sets $I \subseteq [N]$ and values $c \in \{0, 1, -1\}.$

Suppose that $(T,v)$ is semistable relative to $\shL$. Then, by the preceding discussion, $(T,v)$ satisfies some list of conditions of the form
$$ \Omega(v, \mc{H}) \leq c_0(T, \mc{H})q, $$
where $\mc{H}$ is a flag and $c_0(T, \mc{H}) \in \{0, 1, -1\}$.

We show that the conditions described in the theorem statement appear in this way.

Let $\mc{H}$ be Type I. Let the single element of $\mc{H}$ be $H$. We claim there exists a flag completion $\mc{H}'$ of $\mc{H}$ such that the completed flag pair $(\mc{H}, \mc{H}')$ is of Type (P-1A). Indeed, let $\mc{H}''$ be any flag completion of $\mc{H}$, with some associated basis $B''$. Since $0 \neq TH_1$ or $T(k^{N+1}) \subsetneq H$, there is a permutation $B'$ of the basis $B''$ that fixes $H$, such that the matrix of $T$ is not strictly-upper-triangular. Let $\mc{H}'$ be the flag associated to $B'$. Since $TH_1 \subseteq H_1$, the pair $(\mc{H}, \mc{H}')$ is Type (P-1A). Let $\rho \in [N]$ be the value such that $H = \mc{H}'_\rho$. By Proposition \ref{prop_facets}, in the basis $B'$, there is a facet of $\Pi_+(M)$ with supporting halfspace $\{s_\rho \geq 0 \}$. Thus $s_\rho(z)(-\eta) \geq 0$, so by \eqref{eq_omega_functional}, we have $\Omega(v, H) \leq 0$.

Let $\mc{H}$ be Type II. There exists a flag completion $\mc{H}'$ of $\mc{H}$ such that the completed flag pair $(\mc{H}, \mc{H}')$ is of Type (P-1B). Let $I \subseteq [N]$ be the subset of indices corresponding to the inclusion of $\mc{H}$ in $\mc{H}'$. By Proposition \ref{prop_facets}, writing $M$ for the matrix of $T$ in basis $B$, there is a facet of $\Pi_+(M)$ with supporting halfspace $\{s_I \geq -1 \}$. Thus $s_I(z)(-\eta) \geq -q$, so by \eqref{eq_omega_functional}, we have $\Omega(v, \mc{H}) \leq q$.

Let $\mc{H}$ be Type III. There exists a flag completion $\mc{H}'$ of $\mc{H}$ such that the completed flag pair $(\mc{H}, \mc{H}')$ is of Type (P-2B); namely, let $\mc{H}'$ be any completion of 
$$T^{-1}(0) \subsetneq T^{-1}(H_1) \hdots \subsetneq T^{-1}(H_{\gamma - 1}).$$ Let $I \subseteq [N]$ be the subset of indices corresponding to the inclusion of $\mc{H}$ in $\mc{H}'$. By Proposition \ref{prop_facets}, writing $M$ for the matrix of $T$ in basis $B$, there is a facet of $\Pi_+(M)$ with supporting halfspace $\{ s_I \geq 1 \}$. Thus $\Omega(v, \mc{H}) \leq -q.$

We have proved one direction of the theorem; now we prove the other.

For every complete flag $\mc{H}'$, and every subflag $\mc{H}$ of $\mc{H}'$ such that the completed flag pair $(\mc{H}, \mc{H}')$ is Type (P-1A), (P-1B), (P-2A), or (P-2B), we define a quantity $c(T, \mc{H}, \mc{H}')$ by
\[
c(T, \mc{H}, \mc{H'}) =
\begin{cases}
0, & (\mc{H}, \mc{H'}) \text{ is Type (P-1A) for } T, \\
-1, & (\mc{H}, \mc{H'}) \text{ is Type (P-1B) for } T, \\
0, & (\mc{H}, \mc{H'}) \text{ is Type (P-2A) for } T, \\
1, & (\mc{H}, \mc{H'}) \text{ is Type (P-2B) for } T. \\
\end{cases}
\]
We claim that, if \eqref{eq_flag_solo} holds
for all flags $\mc{H}$ of Type I, II, and III, then for all completed flag pairs $(\mc{H}, \mc{H}')$ of Type (P-1A), (P-1B), (P-2A), or (P-2B),
\begin{equation} \label{eq_flag_pairs}
    \Omega(v, \mc{H}) \leq c(T, \mc{H}, \mc{H}') q.
\end{equation}

Indeed, assuming \eqref{eq_flag_pairs}, by Proposition \ref{prop_facets}, for every basis $B$ of $k^{N+1}$, writing $M$ for the matrix of $T$, we have $-\eta(v) \in q\Pi_+(M)$ by \eqref{eq_omega_functional}. Hence $0 \in \Pi_+(T,v)$ for every basis $B$, proving semistability.

Finally, we check that \eqref{eq_flag_solo} implies \eqref{eq_flag_pairs}:
\begin{itemize}
    \item If $(\mc{H}, \mc{H}')$ is Type (P-1A) or (P-2A), then $\mc{H}$ is Type I and $c(T, \mc{H}, \mc{H}') = c(T, \mc{H})$.
    \item 
If $(\mc{H}, \mc{H}')$ is Type (P-1B), then $\mc{H}$ is Type II and $c(T, \mc{H}, \mc{H'}) = c(T, \mc{H})$.
\item
If $(\mc{H}, \mc{H})'$ is Type (P-2B), then $\mc{H}$ is Type III and $c(T, \mc{H}, \mc{H}') = c(T, \mc{H})$.
\end{itemize}
For stability, the same argument works using strict inequalities.
\end{proof}

Our main theorem follows:
\begin{proof}[Proof of Theorem \ref{thm_main}]
In Theorem \ref{thm_main_any_sheaf}, take $m_1 = \hdots = m_n = 1$.
\end{proof}

Finally, we prove two useful facts about the structure of the quotient: it is nonempty (Corollary \ref{cor_stable_nonempty}) and rational (Theorem \ref{thm_rational}).

\begin{proof}[Proof of Corollary \ref{cor_stable_nonempty}]
  Let $T \neq 1$ be invertible with distinct eigenvalues. Then there exists a point $v$ that is not in any $T$-invariant proper linear subspace. Let $v_1 = \hdots = v_n = v$. Let us check that the stability conditions in Theorem \ref{thm_main} hold. Since $T$ is invertible, it is not nilpotent, so there are no Type III flags relative to $T$. The conditions for the Type II flags are met since $q \geq n$. The Type I flags are simply the invariant linear subspaces for $T$, but these do not contain $v$, so the Type I conditions are all met.
\end{proof}

\begin{proof}[Proof of Theorem \ref{thm_rational}]
 The existence and dimension count follow immediately from Corollary \ref{cor_stable_nonempty}. The claim is independent of $\shL$ because varying the sheaf $\shL$ does not change the birational type of the quotient. Let $\mc{S}_{N+1}$ denote the symmetric group on $N + 1$ letters. Given a generic point $(T, v_1, \hdots, v_n) \in \Twip$, consider the $N + 1$ fixed points of $T$ as an element of $(\PP^N)^n / \mc{S}_{N+1}$. There exists a projective transformation that diagonalizes $T$ to some matrix $\diag[\lambda_1, \hdots, \lambda_{N+1}]$ of eigenvalues and takes $v_1$ to $[1: \hdots : 1]$. For each $i$ in the range $1 < i \leq n$, let $v'_i$ be the image of $v_i$. Since there is no ordering on the eigenvectors, a projective transformation with this property is unique up to the action of $\mc{S}_{N+1}$. Thus the multivalued map 
 $$(T, v_1, \hdots, v_n) \dashrightarrow ([\lambda_1 : \hdots : \lambda_{N+1}], v'_2, \hdots, v'_n)$$
 descends to a well-defined rational map
 $$\Twipq \dashrightarrow (\PP^N)^{n} / \mc{S}_{N+1},$$
 where $\mc{S}_{N+1}$ acts diagonally on each factor by permuting coordinates. Then, to prove rationality, it suffices to show that each factor $\PP^N/ \mc{S}_{N+1}$ is rational. The theorem on symmetric functions states that $\A^{N+1} / \mc{S}_{N+1} \cong \A^{N+1}$, so $\PP^N / \mc{S}_{N+1}$ is a weighted projective space, and these are rational.
\end{proof}

\section{Application: one marked point} \label{sect_one_marked_point}

In this section, we work out a special case of Theorem \ref{thm_main}, described in Example \ref{ex_one_marked_point}. We show that, when there is one marked point, our results specialize to a moduli space of Mumford-Suominen \cite{MR0437531}.

Let the dimension $N \geq 1$, let the number of marked points be $n = 1$, and let $q = 1$. With these choices, Theorem \ref{thm_main} describes stability for linear self-maps of $\PP^N$ marked with a single point, embedded in projective space via the Segre embedding $\PP^{N^2 + 2N} \times \PP^N \hookrightarrow \PP^{(N+1)^3 - 1}$. We show that, with these parameters, the stable quotient is isomorphic to $\PP^N$ (Proposition \ref{prop_quotient_is_PN}). Note that any choice of $q \geq 1$ gives the same result; we choose $q = 1$ for concreteness.

First, we collect the most basic consequences of Theorem \ref{thm_main}. Theorem \ref{thm_main} describes (semi)stability by checking the inequality \eqref{eq_main} for each flag of Type I, II, or III. With these parameters, the right side of \eqref{eq_main} is not an integer, and the left side is an integer, so the inequality \eqref{eq_main} is always strict. Thus
$$\mc{X}_{N,1}(1,1)\gits = \mc{X}_{N,1}(1,1)\gitss.$$
Further, according to Corollary \ref{cor_stable_nonempty} and Theorem \ref{thm_rational}, the stable locus is nonempty and the (semi)stable quotient is a rational projective variety of dimension $N$.

To say more about the structure of the quotient, we need the more detailed description of stability from Theorem \ref{thm_main}. With these parameters, Theorem \ref{thm_main} condenses to a simple linear-algebraic condition.

\begin{lemma} \label{lem_one_marked_point}
A marked linear map $(T,v) \in \mc{X}_{N,1}$ is stable relative to $\shO(1,1)$ if and only if $v, T(v), \hdots, T^N(v)$ form a basis of $k^{N+1}$ and $T^{N+1}(v) \neq 0.$
\end{lemma}

\begin{proof}
First, observe that if $T$ has any Type III flags, then inspecting \eqref{eq_main}, the pair $(T,v)$ is unstable. Every nilpotent map admits a Type III flag, so we conclude that if $T$ is nilpotent, then $(T,v)$ is unstable. Second, if $v$ lies on a Type I flag of $T$, then $(T,v)$ is unstable.  The Type II flags impose no further conditions because, for any Type II flag, the left side of \eqref{eq_main} is less than the right side of \eqref{eq_main}. So by Theorem \ref{thm_main}, we have that $(T,v)$ is stable if and only if $T$ is non-nilpotent and $v$ lies on no Type I flag for $T$.

Now we claim that, if $(T,v)$ is stable, then the points $v, T(v), T^2(v), \hdots, T^{N-1}(v), T^N(v)$ are a basis of $k^{N+1}$. Indeed, if not, then there is a $T$-invariant proper subspace containing $v$, thus a Type I flag containing $v$, given by
$$\Span \bigl( T^m (v): 0 \leq m \leq N - 1 \bigr).$$

Conversely, suppose that a pair $(T,v)$ satisfies the condition that $v, \hdots, T^N(v)$ form a basis of $k^{N+1}$. Then the only $T$-invariant subspace containing $v$ is the whole space $k^{N+1}$, and thus $v$ is not contained in any Type I flag. Note further that, for any such $(T,v)$, nilpotence is equivalent to $T^{N+1}(v) = 0.$
\end{proof}

The simple characterization of stability in Lemma \ref{lem_one_marked_point} allows us to explicitly describe the structure of the stable quotient.

\begin{prop} \label{prop_quotient_is_PN}
The stable quotient is isomorphic to a weighted projective space
$$ (\mc{M}_{1,1}^N(1,1))\gits \cong \PP(1,2,\ldots,N+1).$$
\end{prop}
\begin{proof}
We exhibit an explicit isomorphism $\alpha: (\mc{M}_{1,1}^N(1,1))\gits \to \PP(1,2,\ldots,N+1)$. Let $(T,v)$ be a stable marked linear map relative to $\shO(1,1).$
Let $\hat{T} \in \Mat_{N+1}$ be an affine representative of $T$. Since $\hat{T}$ is non-nilpotent (Lemma \ref{lem_one_marked_point}), its characteristic polynomial is of the form
$\hat{T}^{N+1}+a_N\hat{T}^N+\ldots+a_0$, where not all the $a_i$ are $0$. For any $b\in k^*$, the characteristic polynomial of $\frac{1}{b}\hat{T}$ is 
$$\hat{T}^{N+1}+ba_N\hat{T}^N+\ldots+b^{N+1}a_0.$$
Let $\alpha(T,v)$ be the class of $(a_N,\ldots,a_0)$ up to the scaling action of $b$; it does not depend on the choice of representative $\hat{T}$ or on $v$.

Now we define an inverse of $\alpha$. Send the scaling class of $(a_N,\ldots,a_0)$ to the class of $(T,v)$, where $v=[1:0:\ldots:0]$ and
\begin{equation}
T=\begin{bmatrix}
0 &     &       & & -a_0 \\
1 &  0  &       & & -a_1 \\
  &  1  & \ddots & & \vdots \\
  &     & \ddots & 0 &    -a_N     \\
  &     &        & 1      & -a_{N+1} \\
\end{bmatrix}.
\label{eq_lift}
\end{equation}

This is well-defined as a map from $\PP(1,\ldots,N+1)$ since, for any $b \in k^*$, 
\[
\begin{bmatrix}
0 &     &       & & -a_0 b^{N+1} \\
1 &  0  &       & & -a_1 b^N\\
  &  1  & \ddots & & \vdots \\
  &     & \ddots & 0 &    -a_{N-1} b^2     \\
  &     &        & 1      & -a_{N} b \\
\end{bmatrix}
\]
is conjugate to $T$ by $\diag(b^{N},\ldots,b,1)$, which fixes $v$. Since $\PP(1,\ldots,N+1)$ does not contain representatives of $(0,\ldots,0)$, the map $T$ is non-nilpotent, and $v$ is cyclic, so $(T,v)$ is stable (Lemma \ref{lem_one_marked_point}). There is a lift $\hat{T}$ of matrix $T$ with characteristic polynomial $\hat{T}^{N+1}+a_N\hat{T}^N+\ldots+a_0$, by interpreting \eqref{eq_lift} as an affine matrix, so we have an inverse of $\alpha$.
\end{proof}

\begin{remark}
A vector $v$ such that $v, T(v), \hdots, T^N(v)$ forms a basis of $k^{N+1}$ is called a \emph{cyclic vector} of $T$. Thus $\mc{M}_{1,1}^N$ may be described as the moduli space of non-nilpotent linear maps on projective space equipped with a cyclic vector. Mumford and Suominen studied this moduli space, or rather, its affine version, in \cite{MR0437531}. They observe that the stable locus for the moduli problem of unmarked linear maps is empty, but that this deficiency can be addressed by including the data of a cyclic vector. Lemma \ref{lem_one_marked_point} shows that this choice of extra data arises as a natural consequence of the stability theory of marked linear maps.
\end{remark}
\bibliographystyle{amsalpha}
\bibliography{bib}
\end{document}